\documentclass{jmsj2}

\usepackage{amsmath,bm}
\usepackage{amssymb}
\usepackage{amsthm}
\usepackage{mathtools}
\usepackage{cases}
\usepackage{graphicx,color}
\usepackage{float}

\theoremstyle{plain}
\newtheorem{theorem}{Theorem}[section]
\newtheorem{lemma}[theorem]{Lemma}
\newtheorem{cor}[theorem]{Corollary}
\newtheorem{prop}[theorem]{Proposition}

\theoremstyle{definition}
\newtheorem{definition}[theorem]{Definition}
\newtheorem{example}[theorem]{Example}

\theoremstyle{remark}
\newtheorem{rem}[theorem]{Remark}
\newtheorem*{notation}{Notation}

\newcommand{\relmiddle}[1]{\mathrel{}\middle#1\mathrel{}}
\numberwithin{equation}{section}
\allowdisplaybreaks[3]
\DeclareRobustCommand{\qed}{%
  \ifmmode \mathqed
  \else
    \leavevmode\unskip\penalty9999 \hbox{}\nobreak\hfill
    \quad\hbox{\qedsymbol}%
  \fi
}

\newcommand{\E}{\mathbb{E}}

\newcommand{\N}{\mathbb{N}}
\renewcommand{\P}{\mathbb{P}}

\newcommand{\R}{\mathbb{R}}
\newcommand{\T}{\mathbb{T}}
\newcommand{\Z}{\mathbb{Z}}

\newcommand{\cB}{\mathcal{B}}

\newcommand{\cF}{\mathcal{F}}

\newcommand{\cM}{\mathcal{M}}
\newcommand{\cP}{\mathcal{P}}
\newcommand{\cS}{\mathcal{S}}
\newcommand{\cT}{\mathcal{T}}

\renewcommand{\i}{\mathbf{i}}
\renewcommand{\j}{\mathbf{j}}

\newcommand{\p}{\mathbf{p}}

\renewcommand{\a}{\alpha}
\renewcommand{\b}{\beta}
\newcommand{\dl}{\delta}
\newcommand{\eps}{\varepsilon}

\newcommand{\lm}{\lambda}
\newcommand{\om}{\omega}
\newcommand{\Om}{\Omega}
\newcommand{\sg}{\sigma}
\newcommand{\Sg}{\Sigma}
\newcommand{\la}{\langle}
\newcommand{\ra}{\rangle}
\newcommand{\del}{\partial}

\DeclareMathOperator{\bin}{Bin}

\DeclareMathOperator{\dom}{Dom}
\DeclareMathOperator{\im}{Im}

\DeclareMathOperator{\loc}{loc}
\DeclareMathOperator{\po}{Po}
\DeclareMathOperator{\pt}{PT}

\DeclareMathOperator{\sgn}{sgn}

\DeclareMathOperator{\up}{up}

\begin{document}
\title[Law of large numbers for Betti numbers]{Law of large numbers for Betti numbers of homogeneous and spatially independent random simplicial complexes}
\author{Shu \textsc{Kanazawa}}
\address{Kyoto University Institute for Advanced Study, Kyoto University, Japan}
\email{kanazawa.shu.73m@st.kyoto-u.ac.jp}

\begin{abstract}
The Linial--Meshulam complex model is a natural higher-dimensional analog of the Erd\H os--R\'enyi graph model. In recent years, Linial and Peled established a limit theorem for Betti numbers of Linial--Meshulam complexes with an appropriate scaling of the underlying parameter. The present paper aims to extend that result to more-general random simplicial complex models. We introduce a class of homogeneous and spatially independent random simplicial complexes, including the Linial--Meshulam complex model and the random clique complex model as special cases, and we study the asymptotic behavior of their Betti numbers. Moreover, we obtain the convergence of the empirical spectral distributions of their Laplacians.
A key element in the argument is the local weak convergence of simplicial complexes. Inspired by the work of Linial and Peled, we establish the local weak limit theorem for homogeneous and spatially independent random simplicial complexes.
\end{abstract}

\maketitle

\section{Introduction}
Let $K_n$ be the complete graph on $[n]\coloneqq\{1,2,\ldots,n\}$. An Erd\H os--R\'enyi graph is a random subgraph of $K_n$ with $n$ vertices, where each edge in $K_n$ appears independently with a probability $p\in[0,1]$. The probability distribution is denoted by $G(n,p)$. The Erd\H os--R\'enyi graph model has been extensively studied since the early 1960s~(\cite{ER1},~\cite{ER2},~\cite{Gi}) as a typical random graph model. One of the main themes in the study of the Erd\H os--R\'enyi graph model is searching the threshold probability $p$, typically a function of $n$, for some graph property. The behavior of the random graph around the threshold probability should also be studied as a further theme.
Erd\H os and R\'enyi showed that the threshold probability for the appearance of cycles is $p=1/n$. Furthermore, they established a limit theorem for the number of connected components $\xi$ around the threshold probability as follows.
\begin{theorem}[Erd\H os and R\'enyi~{\cite[Section~6]{ER2}}]\label{thm:LT_ERG}
Let $c>0$ be fixed, and let $G_n\sim G(n,p)$ be an Erd\H os--R\'enyi graph with $p=c/n$. Then, for any $\eps>0$, 
\begin{equation}\label{eq:LT_ERG}
\lim_{n\to\infty}\P\biggl(\biggl|\frac{\xi(G_n)}n-\frac1c\sum_{s=1}^\infty\frac{s^{s-2}}{s!}(ce^{-c})^s\biggr|>\eps\biggr)=0. 
\end{equation}
Here, $\xi(G_n)$ denotes the number of connected components of $G_n$. 
\end{theorem}

Recently, there has been growing interest in studying random simplicial complexes as higher-dimensional analogs of random graphs. The systematic study of random simplicial complexes has its origin in the work of Linial and Meshulam~\cite{LM1}. They introduced the $2$-Linial--Meshulam complex model as a higher-dimensional generalization of the Erd\H os--R\'enyi graph model, and the $d$-Linial--Meshulam complex model was studied by Meshulam and Wallach~\cite{LM2}. Let $\triangle_n$ denote the $(n-1)$-dimensional complete complex on $[n]$. For each $d\in\N$, a $d$-Linial--Meshulam complex is a random subcomplex of $\triangle_n$ with all the $(d-1)$-simplices in $\triangle_n$, where each $d$-simplex in $\triangle_n$ appears independently with a probability $p\in[0,1]$. The probability distribution is denoted by $Y_d(n,p)$. Note that a $1$-Linial--Meshulam complex can be naturally regarded as an Erd\H os--R\'enyi graph. 

Since the appearance of cycles in the Erd\H os--R\'enyi graph can be described as the nontriviality of the first homology group, it is natural to seek the threshold probability for the appearance of the $d$th homology group of the $d$-Linial--Meshulam complex. Aronshtam and Linial~\cite{AL} found an upper bound on the threshold probability for the appearance of the $d$th homology group with any field coefficient. Linial and Peled~\cite{LP1} proved that the upper bound is tight as long as the characteristic of the coefficient field is zero. They also studied the asymptotic behavior of the (reduced) Betti numbers of $d$-Linial--Meshulam complexes around the threshold probability. For $c\ge0$, letting $t_{d,c}\in(0,1]$ be the smallest positive root of the equation $t=\exp(-c(1-t)^d)$, we define
\begin{align}\label{eq:h}
h_{d-1}(c)
&\coloneqq\max\biggl\{1-\frac c{d+1},t_{d,c}+ct_{d,c}(1-t_{d,c})^d-\frac c{d+1}\bigl(1-(1-t_{d,c})^{d+1}\bigr)\biggr\}\nonumber\\
&=\begin{cases}
1-\frac c{d+1}													&(0\le c\le c_d),\\
t_{d,c}+ct_{d,c}(1-t_{d,c})^d-\frac c{d+1}\bigl(1-(1-t_{d,c})^{d+1}\bigr)	&(c>c_d). 
\end{cases}
\end{align}
Here, $c_d$ is defined as follows: using the smallest root $x_d\in(0,1]$ of the equation $(d+1)(1-x)+(1+dx)\log x=0$, define $c_d\coloneqq-(\log x_d)/(1-x_d)^d$ for $d\ge2$ and set $c_1\coloneqq1$~(see Remark~1.2 and Appendix~B in~\cite{LP1} for more details). The following theorem is immediately obtained by combining Theorems~1.1 and~1.3 in~\cite{LP1} with the Euler--Poincar\'e formula.
\begin{theorem}[Linial and Peled~\cite{LP1}]\label{thm:LT_B_LMC}
Let $d\in\N$ and $c>0$ be fixed, and let $Y_n\sim Y_d(n,p)$ be a $d$-Linial--Meshulam complex with $p=c/n$. Then, for any $\eps>0$, 
\[
\lim_{n\to\infty}\P\biggl(\biggl|\frac{\b_{d-1}(Y_n)}{n^d}-\frac{h_{d-1}(c)}{d!}\biggr|>\eps\biggr)=0. 
\]
\end{theorem}
\begin{figure}[H]
\centering
\includegraphics[width=10cm,bb=0 0 576 239]{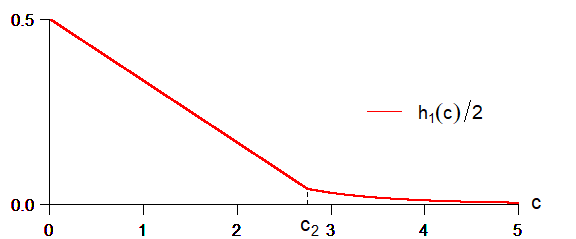}
\caption{Illustration of Theorem~\ref{thm:LT_B_LMC} for $d=2$.}
\label{fig:h1}
\end{figure}
When $d=1$, Theorem~\ref{thm:LT_B_LMC} corresponds to Theorem~\ref{thm:LT_ERG}
since the zeroth (reduced) Betti number is just one less than the number of connected components,
and $h_0(c)$ is identical to the limiting constant in Eq.~\eqref{eq:LT_ERG}. Figure~\ref{fig:h1} illustrates the behavior of the limiting constant $h_{d-1}(c)/d!$ with respect to the parameter $c$ when $d=2$. The aim of this paper is to generalize Theorem~\ref{thm:LT_B_LMC} to more-general random simplicial complex models. 

Kahle~\cite{K14a} introduced another random simplicial complex model called the random clique complex model. Given a simple undirected graph $G$, its \textit{clique complex} is defined as the inclusion-wise maximal simplicial complex among other simplicial complexes whose underlying graphs are identical to $G$. The clique complex of an Erd\H os--R\'enyi graph that follows $G(n,p)$ is called a random clique complex. The probability distribution is denoted by $C(n,p)$. Whenever $k\ge1$, the $k$th Betti number of the random clique complex behaves almost unimodally with respect to the parameter $p$, unlike the case of the Linial--Meshulam complex model~(see Figure~4 in~\cite{K14b} for a numerical experiment by Afra Zomorodian). 

Our first result herein is the counterpart of Theorem~\ref{thm:LT_B_LMC} to the random clique complex model. 
\begin{theorem}\label{thm:LT_B_RCC}
Let $k\ge0$ and $c>0$ be fixed, and let $C_n\sim C(n,p)$ be a random clique complex with $p=(c/n)^{1/(k+1)}$. Then, for any $r\in[1,\infty)$, 
\[
\lim_{n\to\infty}\E\biggl[\biggl|\frac{\b_k(C_n)}{n^{k/2+1}}-\frac{c^{k/2}h_k(c)}{(k+1)!}\biggr|^r\biggr]=0. 
\]
\end{theorem}
\begin{figure}[H]
\centering
\includegraphics[width=10cm,bb=0 0 576 239]{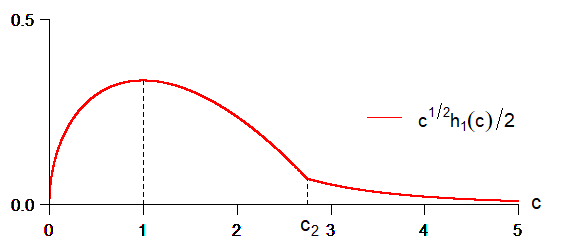}\\
\caption{Illustration of Theorem~\ref{thm:LT_B_RCC} for $k=1$.}
\label{fig:h1clique}
\end{figure}
The $k=0$ case in the theorem above corresponds to Theorem~\ref{thm:LT_ERG} since the underlying graph of $C_n$ is an Erd\H os--R\'enyi graph that follows $G(n,p)$ with $p=c/n$.
Whenever $k\ge1$, the limiting constant $c^{k/2}h_k(c)/(k+1)!$ is unimodal with respect to $c$ as shown in Figure~\ref{fig:h1clique}. Informally speaking, the unimodality comes from the competitive relationship between the effect of creating $k$-dimensional cycles with some $k$-simplices and that of filling them with some $(k+1)$-simplices with an increasing number of total simplices. These two effects complicate the situation and give rise to the critical difference between the Linial--Meshulam complex model and the random clique complex model as seen in Figures~\ref{fig:h1} and~\ref{fig:h1clique}. 

Herein, we introduce two properties, namely, \textit{homogeneity} and \textit{spatial independence}, that are satisfied with both the Linial--Meshulam complex model and the random clique complex model.
As seen in Theorem~\ref{thm:cha_HSI}, these two properties turn out to precisely characterize the multi-parameter random simplicial complex model extensively studied in~\cite{CF16},~\cite{CF17a},~\cite{CF17b},~\cite{CF},~\cite{Fo} (see Section~\ref{sec:HSIRSC} for details).
We then provide a limit theorem for the Betti numbers of homogeneous and spatially independent random simplicial complexes~(Theorem~\ref{thm:LT_B_HSIRSC}~(1)). This result can be regarded as a law of large numbers for the Betti numbers. Applying the result to the Linial--Meshulam complex model and the random clique complex model, we can obtain Theorems~\ref{thm:LT_B_LMC} and~\ref{thm:LT_B_RCC}, respectively, as special cases. 

To prove the main theorem, we use the concept of the \textit{local weak convergence} of simplicial complexes. This concept is a generalization of the Benjamini--Schramm convergence of graphs, introduced by Benjamini and Schramm~\cite{BeS} and Aldous and Steele~\cite{AS}. The local weak convergence is critical for estimating the asymptotic behavior of the Betti numbers. This type of approach has been studied in various contexts~(\cite{AL},~\cite{ALLM},~\cite{E},~\cite{LP1},~\cite{LP2},~\cite{S}). Inspired by those studies, especially the formulation in~\cite{LP2}, we establish the local weak limit theorem for homogeneous and spatially independent random simplicial complexes~(Theorem~\ref{thm:LWLT_HSIRSC}),
where a higher-dimensional generalization of the Poisson branching process arises as a universal limiting object.
Consequently, we also obtain the convergence of the empirical spectral distributions of their Laplacians~(Theorem~\ref{thm:LT_B_HSIRSC}~(2)).

This paper is organized as follows. Section~2 presents some basic concepts related to the cohomology of simplicial complexes and defines  the local weak convergence of simplicial complexes. In Section~3, we describe the homogeneity and spatial independence of random simplicial complexes. In Section~4, we discuss the local weak limit theorem for homogeneous and spatially independent random simplicial complexes. Finally, in Section~5 we state the main limit theorem for the Betti numbers of homogeneous and spatially independent random simplicial complexes and the empirical spectral distributions of their Laplacians. The proof of the main theorem is presented invoking Section~4.

\begin{notation}
Throughout this paper, we use the Bachmann--Landau big-$O$/little-$o$ notation and some related notation with respect to the number of vertices $n$ tending to infinity.
For non-negative functions $f(n)$ and $g(n)$, 
\begin{itemize}
\item $f(n)=O(g(n))$ means that there exists a constant $C\ge0$ such that for sufficiently large $n$, $f(n)\le Cg(n)$ holds,
\item $f(n)=o(g(n))$ means that for any $\eps>0$ and for sufficiently large $n$, $f(n)\le\eps g(n)$ holds,
\item $f(n)=\Om(g(n))$ means that $g(n)=O(f(n))$,
\item $f(n)=\om(g(n))$ means that $g(n)=o(f(n))$,
\item $f(n)\asymp g(n)$ means that $f(n)=O(g(n))$ and $f(n)=\Om(g(n))$, and
\item $f(n)\sim g(n)$ means that $\lim_{n\to\infty}f(n)/g(n)=1$. 
\end{itemize}
For a random variable $X$ and a probability measure $\nu$, we also use the symbol $\sim$. The notation $X\sim \nu$ indicates that the distribution of $X$ coincides with $\nu$. 
Given a topological space $S$, we denote by $\cB_S$ and $\cP_S$ the collection of all Borel sets on $S$ and the set of all Borel probability measures on $S$, respectively. Furthermore, let $C_b(S)$ indicate the set of all bounded continuous real functions on $S$. 
For $a,b\in\R$, we write $a\wedge b=\min\{a,b\}$ and $a\vee b=\max\{a,b\}$. 
\end{notation}

\section{Preliminaries}
\subsection{Simplicial cohomology}
Let $X$ be a collection of finite subsets of a set $V$. $X$ is called a \textit{simplicial complex} on $V$ if $X$ satisfies the following two conditions: (i)~$\{v\}\in X$ for all $v\in V$; (ii)~$\sg\in X$ and $\tau\subset\sg$ together imply that $\tau\in X$.
We often simply say that $X$ is a simplicial complex and let $V(X)$ indicate the vertex set $V$. Note that all simplicial complexes include the empty set. 
Below, we describe some notation and terminology for a given simplicial complex $X$. Each element $\sg\in X$ is called a \textit{simplex} in $X$, and a (strict) subset $\tau$ of $\sg$ is called a (strict) \textit{face} of $\sg$. The \textit{dimension} of $\sg\in X$ is defined by $\dim\sg\coloneqq\#\sg-1$. We call $\sg\in X$ with $\dim\sg=k$ a \textit{$k$-simplex} in $X$. The \textit{dimension} of $X$, denoted by $\dim X$, is defined as the supremum of the dimensions of the simplices in $X$. For $k\ge-1$, let $F_k(X)$ denote the set of all $k$-simplices in $X$, and set $f_k(X)\coloneqq\#(F_k(X))$. 
The \textit{degree} of a $k$-simplex $\tau$ in $X$, denoted by $\deg(X;\tau)$, is defined as the number of $(k+1)$-simplices in $X$ containing $\tau$. $X$ is said to be \textit{locally-finite} if every nonempty simplex in $X$ has a finite degree. Furthermore, $X$ is said to be \textit{finite} when $V(X)$ is a finite set. 
A simplicial complex that is contained in $X$ is called a \textit{subcomplex} of $X$. Given a simplex $\tau$ in $X$, let $K(\tau)$ denote the subcomplex of $X$ consisting of all the faces of $\tau$. For $k\ge-1$, the \textit{$k$-skeleton} of $X$ is defined as a subcomplex $X^{(k)}\coloneqq\bigsqcup_{j=-1}^k F_j(X)$. 

Next, we introduce the concepts of the \textit{simplicial cohomology}. Let $X$ be a simplicial complex on $V$. A sequence $(v_0,v_1,\ldots,v_k)\in V^{k+1}$ is called an \textit{ordered} $k$-simplex in $X$ if $ \{v_0,v_1,\ldots,v_k\}\in F_k(X)$. Let $\Sg_k(X)$ denote the set of all ordered $k$-simplices in $X$. By convention, we set $\Sg_{-1}(X)\coloneqq\{\emptyset\}$. When two ordered simplices can be transformed into each other by an even permutation, they are said to be equivalent. We denote the equivalence class of an ordered simplex $\sg=(v_0,v_1,\ldots,v_k)$ by $\la\sigma\ra$ or $\la v_0,v_1,\ldots,v_k\ra$, and call it an \textit{oriented} $k$-simplex. 
For $k\ge-1$, a map $\varphi\colon\Sg_k(X)\to\R$ is called a \textit{$k$-cochain} of $X$ if $\varphi$ is alternating, that is, if $\varphi((v_{\xi(0)},v_{\xi(1)},\dots,v_{\xi(k)}))=(\sgn \xi)\,\varphi((v_0,v_1,\dots,v_k))$ for any $(v_0,v_1,\dots,v_k)\in\Sg_k(X)$ and permutation $\xi$ on $\{0,1,\dots,k\}$. Let $C^k(X)$ be the $\R$-vector space of all $k$-cochains of $X$. Note that $C^{-1}(X)=\R^{\{\emptyset\}}\simeq\R$. For $k\ge-1$, the $k$th \textit{coboundary map} $d_k\colon C^k(X)\to C^{k+1}(X)$ is defined as the linear extension of
\begin{equation}\label{eq:def_dk}
d_k\varphi(\sg)\coloneqq\sum_{i=0}^{k+1} (-1)^i\varphi((v_0,\ldots,\hat v_i,\ldots,v_{k+1}))
\end{equation}
for $\varphi\in C^k(X)$ and $\sg=(v_0,v_1,\ldots,v_{k+1})\in \Sg_{k+1}(X)$. Here, the hat symbol over $v_i$ indicates that this vertex is deleted from $\sg$. For $k\ge0$, define $Z^k(X)\coloneqq\ker d_k$ and $B^k(X)\coloneqq\im d_{k-1}$. A straightforward calculation gives $d_k\circ d_{k-1}=0$ for all $k\ge0$, that is, $Z^k(X)\supset B^k(X)$. The $k$th \textit{cohomology vector space} with coefficients in $\R$ is defined as $H^k(X)\coloneqq Z^k(X)/B^k(X)$. When $X$ is finite, the dimension of $H^k(X)$ is called the $k$th \textit{Betti number} of $X$, denoted by $\b_k(X)$. 

\subsection{Rooted spectral measure and empirical spectral distribution}
Let $H$ be a Hilbert space with an inner product $(\cdot,\cdot)$, and let $\|\cdot\|$ be the induced norm. A densely defined symmetric operator $L$ on $H$ is said to be \textit{essentially self-adjoint} if the closure of $L$ is self-adjoint. Associated with an essentially self-adjoint operator $L$ and $\varphi\in\dom(L)$ with $\|\varphi\|=1$ is the \textit{spectral measure} $\mu_{L,\varphi}$, which is a unique probability measure on $\R$ such that for all $m\in\N$, 
\[
(L^m\varphi,\varphi)=\int_\R x^m\,d\mu_{L,\varphi}(x). 
\]
In the case when $N\coloneqq\dim H<\infty$, the spectral measure $\mu_{L,\varphi}$ is discrete. Let $\lm_i$ and $\psi_i$~($i=1,2,\ldots,N$) be the eigenvalues of $L$ and the corresponding orthonormal basis of eigenvectors, respectively. Then, a simple calculation gives
\begin{equation}\label{eq:finite_spe}
\mu_{L,\varphi}=\sum_{i=1}^N(\varphi,\psi_i)^2\dl_{\lm_i}. 
\end{equation}
Here, $\dl_{\lm_i}$ is the Dirac measure at $\lm_i$, that is, for any measurable set $A\subset\R$,
\[
\dl_{\lm_i}(A)=\begin{cases}
0	&\text{if $\lm_i\notin A$,}\\
1	&\text{if $\lm_i\in A$.}
\end{cases}
\]
In such a case, we can also consider the \textit{empirical spectral distribution} of $L$: 
\[
\mu_L\coloneqq\frac1N\sum_{i=1}^N\dl_{\lm_i}. 
\]

Now, let $X$ be a locally-finite simplicial complex. For $k\ge-1$, we consider the Hilbert space $\ell^2C^k(X)\coloneqq\{\varphi\in C^k(X)\mid\sum_{\sg\in\Sg_k(X)}\varphi(\sg)^2<\infty\}$ with an inner product
\[
(\varphi,\psi)_k\coloneqq\frac1{(k+1)!}\sum_{\sg\in\Sg_k(X)}\varphi(\sg)\psi(\sg). 
\]
By the linear extension of Eq.~\eqref{eq:def_dk}, we can consider a densely defined operator $d_k^{(2)}\colon\ell^2C^k(X)\to\ell^2C^{k+1}(X)$ whose domain includes all $\varphi\in C^k(X)$ with finite support. The $k$th \textit{up Laplacian}~$L_k^{\up}(X)$ on $\ell^2C^k(X)$ is defined by $L_k^{\up}(X)\coloneqq\bigl(d_{k+1}^{(2)}\bigr)^*\circ d_k^{(2)}$, where $\bigl(d_{k+1}^{(2)}\bigr)^*$ is the adjoint operator of $d_{k+1}^{(2)}$ with respect to the inner products $(\cdot,\cdot)_k$ and $(\cdot,\cdot)_{k+1}$. It is easy to confirm that $L_k^{\up}(X)$ is a densely defined symmetric and non-negative definite operator with respect to the inner product $(\cdot,\cdot)_k$. 
Now, we consider a pair $(X,\tau)$ of $X$ and a $k$-simplex $\tau$ in $X$, namely, a \textit{$k$-rooted simplicial complex}. Let $(e_\tau)_{\tau\in F_k(X)}$ be a canonical orthonormal basis of $\ell^2C^k(X)$. When $L_k^{\up}(X)$ is essentially self-adjoint, the spectral measure associated with $L_k^{\up}(X)$ and $e_\tau$ is called the \textit{rooted spectral measure} of $(X,\tau)$, denoted by $\mu_{(X,\tau)}$. When $X$ is finite, Eq.~\eqref{eq:finite_spe} implies that
\[
\mu_{(X,\tau)}=\sum_{i=1}^{f_k(X)}(e_\tau,\psi_i)_k^2\dl_{\lm_i}, 
\]
where $\lm_i$ and $\psi_i$~($i=1,2,\ldots,f_k(X)$) are the eigenvalues of $L_k^{\up}(X)$ and the corresponding orthonormal basis of eigenvectors, respectively. Therefore, we obtain
\begin{align}
\frac1{f_k(X)}\sum_{\tau\in F_k(X)}\mu_{(X,\tau)}
&=\frac1{f_k(X)}\sum_{\tau\in F_k(X)}\Biggr(\sum_{i=1}^{f_k(X)}(e_\tau,\psi_i)_k^2\dl_{\lm_i}\Biggr)\nonumber\\
&=\frac1{f_k(X)}\sum_{i=1}^{f_k(X)}\Biggl(\sum_{\tau\in F_k(X)}(\psi_i,e_\tau)_k^2\Biggr)\dl_{\lm_i}\nonumber\\
&=\frac1{f_k(X)}\sum_{i=1}^{f_k(X)}\dl_{\lm_i}\nonumber\\
&=\mu_{L_k^{\up}(X)}, \label{eq:SpeMeas}
\end{align}
which implies that the empirical spectral distribution is the spatial average of all the rooted spectral measures. In other words, the rooted spectral measure can be regarded as the local contribution of the root to the empirical spectral distribution. 

\subsection{Local weak convergence of simplicial complexes}
For any $k$-rooted simplicial complexes $(X,\tau)$ and $(X',\tau')$, the equivalence $(X,\tau)\simeq(X',\tau')$ means that $(X,\tau)$ and $(X',\tau')$ are root-preserving simplicial isomorphic. The equivalence class of $(X,\tau)$ is denoted by $[X,\tau]$. Given a $k$-rooted simplicial complex $(X,\tau)$, we define a nondecreasing sequence $(X_l)_{l=0}^\infty$ of subcomplexes of $X$ iteratively: 
\[
X_0\coloneqq K(\tau)\quad\text{and}\quad X_l\coloneqq X_{l-1}\cup\bigcup_{\sg\in B_l}K(\sg)\quad\text{for }l\ge1,
\]
where $B_l$ is the set of all simplices in $X$ containing at least one $k$-simplex in $X_{l-1}$. A simplex $\sg$ in $X$ is said to be of distance $l$ from $\tau$ if $\sg\in X_l\setminus X_{l-1}$. We additionally set $X_\infty\coloneqq\bigcup_{l=0}^\infty X_l$. We then define $k$-rooted simplicial complexes $(X,\tau)_l\coloneqq(X_l,\tau)$ for $l\ge0$, and $X(\tau)=(X_\infty,\tau)$. For the simplicity, let us denote the equivalence class of $(X,\tau)_l$ by $[X,\tau]_l$ for $l\ge0$, and similarly that of $X(\tau)$ by $X[\tau]$. 

Let $\cS_k$ denote the set of all equivalence classes $[X,\tau]$ such that $X$ is locally-finite and $X(\tau)=(X,\tau)$. We define a metric $d_{\loc}$ on $\cS_k$ by letting the distance between $[X_1,\tau_1]$ and $[X_2,\tau_2]$ be $2^{-L}$, where $L$ is the supremum of those $l\in\Z_{\ge0}$ such that $(X_1,\tau_1)_l\simeq(X_2,\tau_2)_l$. Here, we set $2^{-\infty}=0$ by convention.
This metric is called the \textit{local distance}, and the convergence with respect to $d_{\loc}$ is called the \textit{local convergence}. This makes $(\cS_k,d_{\loc})$ into a complete separable metric space. Furthermore, it is easy to verify that $d_{\loc}$ is an ultrametric. 
We say that a sequence $(\mu_n)_{n=1}^\infty$ in $\cP_{\cS_k}$ converges weakly to $\mu\in\cP_{\cS_k}$ if
\[
\lim_{n\to\infty}\int_{\cS_k}g\,d\mu_n=\int_{\cS_k}g\,d\mu
\]
for all $g\in C_b(\cS_k)$. Since $(\cS_k,d_{\loc})$ is a separable ultrametric space, every open set of $\cS_k$ is expressed by a disjoint union of a countable number of open balls. Therefore, it is easy to confirm that the weak convergence in $\cP_{\cS_k}$ is characterized by the convergence of mass on all the open balls as follows.
\begin{lemma}\label{lem:chara_LWC}
Let $(\mu_n)_{n=1}^\infty$ be a sequence in $\cP_{\cS_k}$, and let $\mu\in\cP_{\cS_k}$. Then, $\mu_n$ converges weakly to $\mu$ if and only if, for any $[X,\tau]\in\cS_k$ and $r>0$, 
\[
\lim_{n\to\infty}\mu_n(B([X,\tau],r))=\mu(B([X,\tau],r)). 
\]
Here, $B([X,\tau],r)$ indicates the open ball of radius $r$ centered at $[X,\tau]$. 
\end{lemma}

On the other hand, we are often interested in a sequence of non-rooted finite simplicial complexes. Given a finite simplicial complex $X$ with $\dim X\ge k$, we define a probability measure $\lm_k(X)$ on $\cS_k$ by
\[
\lm_k(X)\coloneqq\frac1{f_k(X)}\sum_{\tau\in F_k(X)}\dl_{X[\tau]}. 
\]
Here, $\dl_{X[\tau]}$ means the Dirac measure at $X[\tau]\in\cS_k$. The probability measure $\lm_k(X)\in\cP_{\cS_k}$ can be regarded as the distribution of the local structure around a uniformly chosen $k$-simplex in $X$. Using this notation, we define the \textit{local weak convergence} of finite simplicial complexes. 
\begin{definition}\label{def:LWC}
Let $(X_n)_{n=1}^\infty$ be a sequence of finite simplicial complexes with $\dim X_n\ge k$, and let $\nu\in\cP_{\cS_k}$. We say that $X_n$ \textit{converges locally weakly} to $\nu$ as $n\to\infty$ if $\lm_k(X_n)$ converges weakly to $\nu$ as $n\to\infty$. 
\end{definition}
\begin{rem}
From Lemma~\ref{lem:chara_LWC}, Definition~\ref{def:LWC} is equivalent to stating that for any $[X',\tau']\in\cS_k$ and $l\in\N$,
\[
\lim_{n\to\infty}\frac{\#\{\tau\in F_k(X_n)\mid(X_n,\tau)_l\simeq(X',\tau')_l\}}{f_k(X_n)}=\nu(\{[\a]\in\cS_k\mid\a_l\simeq(X',\tau')_l\}).
\]
\end{rem}

\section{Homogeneous and spatially independent random simplicial complex}\label{sec:HSIRSC}
Throughout this section, let $n\in\N$ be fixed. Recall that $\triangle_n=2^{[n]}$ indicates the complete complex on $[n]$. We denote the set of all subcomplexes of $\triangle_n$ by $S_n$ and consider $S_n$-valued random variables, namely, \textit{random subcomplexes} of $\triangle_n$, that are defined on a probability space $(\Om,\cF,\P)$.

\subsection{Homogeneity and spatial independence}
The permutation group on $[n]$ naturally acts on $S_n$. Indeed, for any permutation $g$ on $[n]$ and $Y=\{\sg_i\}_i\in S_n$, we define $gY\coloneqq\{g(\sg_i)\}_i\in S_n$, where $g(\sg_i)$ denotes the image of $\sg_i\subset[n]$ under $g$.
\begin{definition}
We say that a random subcomplex $X$ of $\triangle_n$ is \textit{homogeneous} if $X$ and $gX$ have the same distribution for any permutation $g$ on $[n]$.
\end{definition}
\begin{definition}
We
say that a random subcomplex $X$ of $\triangle_n$ is \textit{spatially independent} if, for any $Y_1,Y_2\in S_n$, 
\begin{equation}\label{eq:D0_SI}
\P(Y_1\cup Y_2\subset X)\P(Y_1\cap Y_2\subset X)=\P(Y_1\subset X)\P(Y_2\subset X). 
\end{equation}
\end{definition}
The following lemma gives some characterizations of the spatial independence.
\begin{lemma}
For a random subcomplex $X$ of $\triangle_n$, the following are equivalent$:$
\begin{enumerate}
\item $X$ is spatially independent$;$
\item For any $Y_1,Y_2\in S_n$ with $\P(Y_2\subset X)>0$,
\begin{equation}\label{eq:D1_SI}
\P(Y_1\subset X\mid Y_2\subset X)=\P(Y_1\subset X\mid Y_1\cap Y_2\subset X);
\end{equation}
\item For any $Y_1,Y_2\in S_n$ with $\P(Y_1\cap Y_2\subset X)>0$, 
\[
\P(Y_1\cup Y_2\subset X\mid Y_1\cap Y_2\subset X)=\P(Y_1\subset X\mid Y_1\cap Y_2\subset X)\P(Y_2\subset X\mid Y_1\cap Y_2\subset X);
\]
\item For any $Y_1,Y_2\in S_n$ and $Z\in S_n$ such that $Y_1\cap Y_2\subset Z$ and $\P(Z\subset X)>0$, 
\begin{equation}\label{eq:D3_SI}
\P(Y_1\cup Y_2\subset X\mid Z\subset X)=\P(Y_1\subset X\mid Z\subset X)\P(Y_2\subset X\mid Z\subset X). 
\end{equation}
\end{enumerate}
\end{lemma}
\begin{proof}
(4)~$\Rightarrow$~(3) is trivial by taking $Z=Y_1\cap Y_2$. It is also easy to verify that (1), (2), and (3) are equivalent since
\begin{align*}
\P(Y_1\subset X\mid Y_2\subset X)&=\frac{\P(Y_1\subset X,Y_2\subset X)}{\P(Y_2\subset X)}=\frac{\P(Y_1\cup Y_2\subset X)}{\P(Y_2\subset X)},\\
\P(Y_i\subset X\mid Y_1\cap Y_2\subset X)&=\frac{\P(Y_i\subset X,Y_1\cap Y_2\subset X)}{\P(Y_1\cap Y_2\subset X)}=\frac{\P(Y_i\subset X)}{\P(Y_1\cap Y_2\subset X)}\quad\text{for $i=1,2$,}
\shortintertext{and}
\P(Y_1\cup Y_2\subset X\mid Y_1\cap Y_2\subset X)&=\frac{\P(Y_1\cup Y_2\subset X,Y_1\cap Y_2\subset X)}{\P(Y_1\cap Y_2\subset X)}=\frac{\P(Y_1\cup Y_2\subset X)}{\P(Y_1\cap Y_2\subset X)}.
\end{align*}
For (1)~$\Rightarrow$~(4), suppose that $Y_1,Y_2\in S_n$ and $Z\in S_n$ such that $Y_1\cap Y_2\subset Z$ and $\P(Z\subset X)>0$. Since $Y_1\cap(Y_2\cup Z)=Y_1\cap Z$, the spatial independence of $X$ implies that
\[
\P(Y_1\cup(Y_2\cup Z)\subset X)\P(Y_1\cap Z\subset X)=\P(Y_1\subset X)\P(Y_2\cup Z\subset X). 
\]
Furthermore, 
\[
\P(Y_1\cup Z\subset X)\P(Y_1\cap Z\subset X)=\P(Y_1\subset X)\P(Z\subset X). 
\]
Therefore, noting that $\P(Y_1\cap Z\subset X)\ge\P(Z\subset X)>0$, we have
\begin{align*}
\P(Y_1\cup Y_2\cup Z\subset X)\P(Z\subset X)&=\frac{\P(Y_1\subset X)\P(Z\subset X)}{\P(Y_1\cap Z\subset X)}\P(Y_2\cup Z\subset X)\\
&=\P(Y_1\cup Z\subset X)\P(Y_2\cup Z\subset X). 
\end{align*}
The conclusion is obtained by dividing both sides of the above equation by $\P(Z\subset X)^2$.
\end{proof}
By a simple calculation, Eq.~\eqref{eq:D1_SI} implies that for any $Y\in S_n$ and $Z_1\subset Z_2\in S_n$ with $\P(Z_2\subset X)>0$, 
\begin{equation}\label{eq:M_CP}
\P(Y\subset X\mid Z_1\subset X)\le\P(Y\subset X\mid Z_2\subset X). 
\end{equation}

\subsection{Multi-parameter random simplicial complex model}
The two intrinsic properties---homogeneity and spatial independence---equally characterize a well-established model, the so-called multi-parameter random simplicial complex, which was extensively studied in~\cite{CF16},~\cite{CF17a},~\cite{CF17b},~\cite{CF},~\cite{Fo}.
%
Given a multi-parameter $\p=(p_0,p_1,\ldots,p_{n-1})$ with $p_i\in[0,1]$ for every $i=0, 1,\ldots, n-1$, a multi-parameter random simplicial complex with $\p=(p_0,p_1,\ldots,p_{n-1})$ is constructed as follows.
We start with vertex set $[n]$ and retain each vertex independently with probability $p_0$. Each edge with both end points retained appears independently with probability $p_1$. Iteratively, for $i=2,3,\dots,n-1$, each $i$-simplex in $\triangle_n$ whose all strict faces were included before the current step appears independently with probability $p_i$. The resulting random simplicial complex is called a multi-parameter random simplicial complex with multi-parameter $\p=(p_0,p_1,\ldots,p_{n-1})$. We denote its probability distribution by $X(n,\p)$.

Let $X\sim X(n,\p)$ be a multi-parameter random simplicial complex, and let $Y\in S_n$ be fixed. A nonempty simplex $\sg$ in $\triangle_n$ is called an \textit{external} simplex in $Y$ if $\sg\notin Y$ and $\del\sg\subset Y$. Here, $\del\sg$ indicates the simplicial complex consisting of all the strict faces of $\sg$. Let $E_k(Y)$ indicate the set of all external $k$-simplices in $Y$, and set $e_k(Y)\coloneqq\#(E_k(Y))$. Then, 
\begin{equation}\label{eq:P_XequalY}
\P(X=Y)=\prod_{i=0}^{n-1}\P\bigl(X^{(i)}=Y^{(i)}\mid X^{(i-1)}=Y^{(i-1)}\bigr)=\prod_{i=0}^{n-1}p_i^{f_i(Y)}(1-p_i)^{e_i(Y)}. 
\end{equation}
The homogeneity of $X$ follows from this equation. Furthermore, $X$ is spatially independent. Indeed, noting that
\begin{equation}\label{eq:P_XsubsetY}
\P(Y\subset X)=\prod_{i=0}^{n-1}\P\bigl(Y^{(i)}\subset X^{(i)}\mid Y^{(i-1)}\subset X^{(i-1)}\bigr)=\prod_{i=0}^{n-1}p_i^{f_i(Y)}, 
\end{equation}
we have
\begin{align*}
\P(Y_1\cup Y_2\subset X)\P(Y_1\cap Y_2\subset X)
&=\prod_{i=0}^{n-1}p_i^{f_i(Y_1\cup Y_2)+f_i(Y_1\cap Y_2)}\\
&=\prod_{i=0}^{n-1}p_i^{f_i(Y_1)+f_i(Y_2)}
=\P(Y_1\subset X)\P(Y_2\subset X)
\end{align*}
for any $Y_1,Y_2\in S_n$.
In fact, the multi-parameter random simplicial complex model is characterized by the homogeneity and spatial independence as follows.
\begin{theorem}\label{thm:cha_HSI}
All multi-parameter random simplicial complexes are homogeneous and spatially independent. Moreover, for any homogeneous and spatially independent random subcomplex $X$ of $\triangle_n$, there exists a multi-parameter $\p=(p_0,p_1,\ldots,p_{n-1})$ such that $X\sim X(n,\p)$. 
\end{theorem}
\begin{proof}
The first conclusion is clear from the discussion above. Hence, for the second conclusion, let $X$ be a homogeneous and spatially independent random subcomplex of $\triangle_n$. We define a parameter $\p=(p_0,p_1,\ldots,p_{n-1})$ by
\[
p_k\coloneqq\begin{cases}
\P([k+1]\in X\mid\del[k+1]\subset X)	&\text{if $\P(\del[k+1]\subset X)>0$, }\\
0										&\text{otherwise. }
\end{cases}
\]
Let $Y\in S_n$ be fixed. Then,
\begin{align}\label{eq:chara0}
\P(X=Y)
&=\prod_{i=0}^{n-1}\P\bigl(X^{(i)}=Y^{(i)}\mid X^{(i-1)}=Y^{(i-1)}\bigr)\nonumber\\
&=\prod_{i=0}^{n-1}\P\left(\bigcap_{\sg\in F_i(Y)}\{\sg\in X\}\cap\bigcap_{\tau\in E_i(Y)}\{\tau\notin X\}\relmiddle| X^{(i-1)}=Y^{(i-1)}\right). 
\end{align}
Now, let $0\le i\le n-1$ be fixed. Noting that
\[
\bigl\{X^{(i-1)}=Y^{(i-1)}\bigr\}=\bigl\{Y^{(i-1)}\subset X\bigr\}\cap\bigcap_{\tau\in\triangle_n^{(i-1)}\setminus Y}\{\tau\notin X\}, 
\]
we obtain
\begin{align}\label{eq:chara1}
&\P\Biggl(\bigcap_{\sg\in F_i(Y)}\{\sg\in X\}\cap\bigcap_{\tau\in E_i(Y)}\{\tau\notin X\}\cap\bigl\{X^{(i-1)}=Y^{(i-1)}\bigr\}\Biggr)\nonumber\\
&=\P\Biggl(\Biggl(\bigl\{Y^{(i-1)}\subset X\bigr\}\cap\bigcap_{\sg\in F_i(Y)}\{\sg\in X\}\Biggr)\cap\bigcap_{\tau\in E_i(Y)\sqcup(\triangle_n^{(i-1)}\setminus Y)}\{\tau\notin X\}\Biggr)\nonumber\\
&=\sum_{k=0}^{e_i(Y)}\sum_{l=0}^{\#(\triangle_n^{(i-1)}\setminus Y)}(-1)^{k+l}\sum_{\substack{S\subset E_i(Y)\\ \#S=k}}\sum_{\substack{T\subset\triangle_n^{(i-1)}\setminus Y\\ \#T=l}}p(S,T). 
\end{align}
In the last line, we use the inclusion-exclusion principle.
Here, $p(S,T)$ is defined by
\begin{align*}
p(S,T)
&\coloneqq\P\Biggl(\Biggl(\bigl\{Y^{(i-1)}\subset X\bigr\}\cap\bigcap_{\sg\in F_i(Y)}\{\sg\in X\}\Biggr)\cap\bigcap_{\tau\in S\sqcup T}\{\tau\in X\}\Biggr)\\
&=\P\Biggl(\bigcap_{\sg\in F_i(Y)\sqcup S}\{\sg\in X\}\cap\bigl\{Y^{(i-1)}\subset X\bigr\}\cap\bigcap_{\tau\in T}\{\tau\in X\}\Biggr)\\
&=\P\Biggl(\bigcup_{\sg\in F_i(Y)\sqcup S}K(\sg)\cup\biggl(Y^{(i-1)}\cup\bigcup_{\tau\in T}K(\tau)\biggr)\subset X\Biggr). 
\end{align*}
Now, we set
\[
Z\coloneqq\bigcup_{\sg\in F_i(Y)\cup E_i(Y)}\del\sg. 
\]
Then, for all $S\subset E_i(Y)$ and $T\subset\triangle_n^{(i-1)}\setminus Y$, 
\[
\biggl(\bigcup_{\sg\in F_i(Y)\sqcup S}K(\sg)\biggr)\cap\biggl(Y^{(i-1)}\cup\bigcup_{\tau\in T}K(\tau)\biggr)\subset Z. 
\]
Whenever $\P(Z\subset X)>0$, by applying Eq.~\eqref{eq:D3_SI} iteratively,
\begin{align*}
&\P\left(\bigcup_{\sg\in F_i(Y)\sqcup S}K(\sg)\cup\Biggl(Y^{(i-1)}\cup\bigcup_{\tau\in T}K(\tau)\Biggr)\subset X\relmiddle|Z\subset X\right)\\
&=\P\left(\bigcup_{\sg\in F_i(Y)\sqcup S}K(\sg)\subset X\relmiddle|Z\subset X\right)\P\left(\Biggl(Y^{(i-1)}\cup\bigcup_{\tau\in T}K(\tau)\Biggr)\subset X\relmiddle|Z\subset X\right)\\
&=\Biggl(\prod_{\sg\in F_i(Y)\sqcup S}\P(K(\sg)\subset X\mid Z\subset X)\Biggr)\P\left(\Biggl(Y^{(i-1)}\cup\bigcup_{\tau\in T}K(\tau)\Biggr)\subset X\relmiddle|Z\subset X\right)\\
&=\frac1{\P(Z\subset X)}\Biggl(\prod_{\sg\in F_i(Y)\sqcup S}\P(K(\sg)\subset X\mid\del\sg\subset X)\Biggr)\P\Biggl(\Biggl(Y^{(i-1)}\cup\bigcup_{\tau\in T}K(\tau)\Biggr)\subset X\Biggr)\\
&=\frac{p_i^{f_i(Y)+\# S}}{\P(Z\subset X)}\P\Biggl(\bigl\{Y^{(i-1)}\subset X\bigr\}\cap\bigcap_{\tau\in T}\{\tau\in X\}\Biggr). 
\end{align*}
In the fourth line, we use Eq.~\eqref{eq:D1_SI} and $Z\subset Y^{(i-1)}$. Thus, we obtain
\begin{equation}\label{eq:chara2}
p(S,T)=p_i^{f_i(Y)+\#S}\P\Biggl(\bigl\{Y^{(i-1)}\subset X\bigr\}\cap\bigcap_{\tau\in T}\{\tau\in X\}\Biggr). 
\end{equation}
When $\P(Z\subset X)=0$, Eq.~\eqref{eq:chara2} is easily verified from $Z\subset Y^{(i-1)}$. Combining Eqs.~\eqref{eq:chara1} and~\eqref{eq:chara2} gives
\begin{align*}
&\P\Biggl(\bigcap_{\sg\in F_i(Y)}\{\sg\in X\}\cap\bigcap_{\sg\in E_i(Y)}\{\sg\notin X\}\cap\bigl\{X^{(i-1)}=Y^{(i-1)}\bigr\}\Biggr)\\
&=\sum_{k=0}^{e_i(Y)}\sum_{l=0}^{\#(\triangle_n^{(i-1)}\setminus Y)}(-1)^{k+l}\sum_{\substack{S\subset E_i(Y)\\ \#S=k}}\sum_{\substack{T\subset\triangle_n^{(i-1)}\setminus Y\\ \#T=l}}p_i^{f_i(Y)+k}\P\biggl(\bigl\{Y^{(i-1)}\subset X\bigr\}\cap\bigcap_{\tau\in T}\{\tau\in X\}\biggr)\\
&=p_i^{f_i(Y)}(1-p_i)^{e_i(Y)}\sum_{l=0}^{\#(\triangle_n^{(i-1)}\setminus Y)}(-1)^l\sum_{\substack{T\subset\triangle_n^{(i-1)}\setminus Y\\ \#T=l}}\P\Biggl(\bigl\{Y^{(i-1)}\subset X\bigr\}\cap\bigcap_{\tau\in T}\{\tau\in X\}\Biggr)\\
&=p_i^{f_i(Y)}(1-p_i)^{e_i(Y)}\P\Biggl(\bigl\{Y^{(i-1)}\subset X\bigr\}\cap\bigcap_{\tau\in\triangle_n^{(i-1)}\setminus Y}\{\tau\notin X\}\Biggr)\\
&=p_i^{f_i(Y)}(1-p_i)^{e_i(Y)}\P\bigl(X^{(i-1)}=Y^{(i-1)}\bigr). 
\end{align*}
In the fourth line, we again use the inclusion-exclusion principle.
From Eq.~\eqref{eq:chara0}, we obtain
\[
\P(X=Y)=\prod_{i=0}^{n-1}p_i^{f_i(Y)}(1-p_i)^{e_i(Y)}, 
\]
which implies that the distribution of $X$ is identical to $X(n,\p)$ from Eq.~\eqref{eq:P_XequalY}. 
\end{proof}
It is worth noting that homogeneous and spatially independent random subcomplexes of $\triangle_n$ require no a priori parameters; however, they are essentially determined by probability parameters $p_0,p_1,\ldots,p_{n-1}$ via Theorem~\ref{thm:cha_HSI}. Hence, we can consider various examples of homogeneous and spatially independent random subcomplexes of $\triangle_n$ by choosing the multi-parameter $\p$ of the multi-parameter random simplicial complex model.
\begin{example}[$d$-Linial--Meshulam complex]\label{ex:dLMC}
Let $1\le d<n$ and $p\in[0,1]$ be fixed. We define $\p=(p_0,p_1,\ldots,p_{n-1})$ by
\[
p_i\coloneqq\begin{cases}
1		&(0\le i\le d-1), \\
p 		&(i=d), \\
0 		&(d+1\le i\le n-1). 
\end{cases}
\]
The corresponding random simplicial complex is a $d$-Linial--Meshulam complex that follows $Y_d(n,p)$. When $d=1$, this can be regarded as an Erd\H os--R\'enyi graph that follows $G(n,p)$. 
\end{example}
\begin{example}[Random $d$-clique complex]\label{ex:RdCC}
Let $1\le d<n$ and $p\in[0,1]$ be fixed. We define $\p=(p_0,p_1,\ldots,p_{n-1})$ by
\[
p_i\coloneqq\begin{cases}
1		&(0\le i\le d-1), \\
p 		&(i=d), \\
1 		&(d+1\le i\le n-1). 
\end{cases}
\]
The corresponding random simplicial complex is called a \textit{random $d$-clique complex}~(see~\cite{Ta} and~\cite[Example~3.2]{HK}). We denote its probability distribution by $C_d(n,p)$. When $d=1$, we obtain a random clique complex that follows $C(n,p)$. 
\end{example}

\section{Local weak limit theorem for random simplicial complexes}
\subsection{Statement of the result}
In this section, we consider homogeneous and spatially independent random subcomplexes of $\triangle_n$, and we study their local weak convergence as $n$ tends to infinity. To state the theorem, we describe a higher-dimensional generalization of the Galton--Watson tree with Poisson offspring distribution~(see also~\cite[Section~3]{ALLM} and~\cite[Section~3]{LP2}).

For $k\ge0$, a $k$-rooted simplicial complex $(T,\tau)$ is called a \textit{$k$-rooted tree} if $T$ can be constructed by the following process: we start with $K(\tau)$; at each step $l=0,1,2,\ldots$, to every $k$-simplex $\tau'$ of distance $l$ from $\tau$, we pick a non-negative number $m(\tau')$ of the new vertices $v_1,v_2,\ldots,v_{m(\tau')}$, and add the $(k+1)$-simplices $\tau'\cup\{v_1\},\tau'\cup\{v_2\},\ldots,\tau'\cup\{v_{m(\tau')}\}$ to the simplicial complex constructed before the current step. A simplicial complex $T$ is called a \textit{$(k+1)$-tree} if $(T,\tau)$ is a $k$-rooted tree for some $k$-simplex $\tau$ in $T$.

When we sample each number $m(\tau')$ in the generative process of a $k$-rooted tree from the Poisson distribution $\po_c$ with parameter $c\ge0$ independently of any others, the resulting object is called a \textit{$k$-rooted Poisson tree} with parameter $c$. In what follows, let $(\pt_k(c),\tau_o)$ indicate a $k$-rooted Poisson tree that is defined on a probability space $(\Om',\cF',\P')$. The expectation with respect to $\P'$ is denoted by $\E'$. It is easy to confirm that $[\pt_k(c),\tau_o]$ is an $\cS_k$-valued random variable, and we denote its probability distribution by $\nu_k(c)\in\cP_{\cS_k}$.

Next, we introduce two types of parameters for a given homogeneous and spatially independent random subcomplex $X_n$ of $\triangle_n$. For $k\ge-1$, 
\begin{equation}\label{eq:para}
q_k\coloneqq\P(\tau\in X_n)\quad\text{and}\quad r_k\coloneqq\begin{cases}
\P(\sg\in X_n\mid\tau\in X_n)	&(q_k>0),\\
0							&(q_k=0). 
\end{cases}
\end{equation}
Here, $\tau$ and $\sg$ are arbitrarily fixed $k$- and $(k+1)$-simplices in $\triangle_n$, respectively, such that $\tau\subset\sg$. The parameters $q_k$ and $r_k$ are well-defined from the homogeneity of $X_n$. It is easy to confirm that $q_{k+1}=q_kr_k$ for $k\ge-1$ and that $q_k$ is nonincreasing with respect to $k$.
Note that when we consider a multi-parameter random simplicial complex with $\p=(p_0,p_1,\ldots,p_{n-1})$, the corresponding $q_k$'s and $r_k$'s are given by
\begin{equation}\label{eq:paraMP}
q_k=\begin{cases}
1									&(k=-1),\\
\prod_{i=0}^k p_i^{\binom{k+1}{i+1}}		&(0\le k\le n-1),\\
0									&(k\ge n)
\end{cases}
\quad\text{and}\quad
r_k=\begin{cases}
p_0									&(k=-1),\\
\prod_{i=0}^{k+1}p_i^{\binom{k+1}i}		&(0\le k\le n-2),\\
0									&(k\ge n-1)
\end{cases}
\end{equation}
using Eq.~\eqref{eq:P_XsubsetY} together with the fact that $q_{k+1}=q_kr_k$ for $k\ge-1$.
The following local weak limit theorem for homogeneous and spatially independent random subcomplexes of $\triangle_n$ is the main result in this section.
\begin{theorem}\label{thm:LWLT_HSIRSC}
Let $k\ge0$ and $c>0$ be fixed, and let $X_n$ be a homogeneous and spatially independent random subcomplex of $\triangle_n$. If $n^{k+1}q_k=\om(1)$ and $nr_k\sim c$, then for any open set $U\subset\cP_{\cS_k}$ such that $\nu_k(c)\in U$, 
\[
\lim_{n\to\infty}\P(\lm_k(X_n)\in U\mid\dim X_n\ge k)=1. 
\]
In other words, $X_n$ under $\P(\cdot\mid\dim X_n\ge k)$ converges locally weakly to $\nu_k(c)$ in distribution as $n\to\infty$. 
\end{theorem}
\begin{rem}
We start with a few remarks on the assumption of Theorem~\ref{thm:LWLT_HSIRSC}.
\begin{enumerate}
\item Consider the $d$-Linial--Meshulam complex that follows $Y_d(n,p)$. Letting $k=d-1$, we have $q_k=1$ and $r_k=p$ for $n\ge d+1$ using Eq.~\eqref{eq:paraMP} with Example~\ref{ex:dLMC}. Thus, the assumptions $n^{k+1}q_k=\om(1)$ and $nr_k\sim c$ are satisfied for $p\sim c/n$.
\item Consider the random $d$-clique complex that follows $C_d(n,p)$. Letting $k\ge d-1$, we have $q_k=p^{\binom{k+1}{d+1}}$ and $r_k=p^{\binom{k+1}d}$ for $n\ge k+2$ using Eq.~\eqref{eq:paraMP} with Example~\ref{ex:RdCC}.
Here, $\binom d{d+1}=0$ by convention.
Thus, the assumptions $n^{k+1}q_k=\om(1)$ and $nr_k\sim c$ are satisfied for $p\sim(c/n)^{1/\binom{k+1}d}$.
\item Consider the multi-parameter random simplicial complex with $\p=(p_0,p_1,\ldots,p_{n-1})$. From Eq.~\eqref{eq:paraMP}, the assumptions $n^{k+1}q_k=\om(1)$ and $nr_k\sim c$ are given as follows:
\[
n^{k+1}\prod_{i=0}^k p_i^{\binom{k+1}{i+1}}=\om(1)\quad\text{and}\quad n\prod_{i=0}^{k+1}p_i^{\binom{k+1}i}\sim c.
\]
These conditions roughly correspond to an area in the space of multi-parameters, whereas in the cases~(1) and~(2) above (where there is a single probability parameter $p$), the assumptions are represented by the numerical conditions of $p$.
The area above is closely related to the notion of \textit{critical dimension} of multi-parameter random simplicial complexes, introduced by Costa and Farber~\cite{CF17b}, as we shall see in Section~\ref{ssec:mainBetti}.
\end{enumerate}
\end{rem}
We devote the rest of this section to proving Theorem~\ref{thm:LWLT_HSIRSC}. The proof of Theorem~\ref{thm:LWLT_HSIRSC} is based on an exploration of a given simplicial complex from a selected simplex. The exploration is a higher-dimensional analog of the breadth-first traversal of a given graph. In what follows in this section, let $k\ge0$ be fixed, and let $X_n$ be a homogeneous and spatially independent random subcomplex of $\triangle_n$. 

\subsection{Breadth-first traversal of simplicial complexes}
Given a subcomplex $X$ of $\triangle_n$ and a $k$-simplex $\tau$ in $X$, we start traversing $X$ layer-wise from $\tau$. More precisely, for each step $i\ge0$, we build a $(k+1)$-tree $T_i$, a current $k$-simplex $\tau_i$, and a bijective map $\varphi_i$ from a subset $W_i$ of $\N^f\coloneqq\{\emptyset\}\sqcup\bigsqcup_{m=1}^\infty\N^m$ to $F_k(T_i)$, iteratively. Here, we equip $\N^f$ with a total order as follows: for two elements $\i=(i_1,i_2,\ldots,i_m)$ and $\i'=(i'_1,i'_2,\ldots,i'_{m'})$, we set $\i<\i'$ if $m<m'$ or if $m=m'$ and there exists $l=1,2,\ldots,m$ such that $(i_1,i_2,\ldots,i_{l-1})=(i'_1,i'_2,\ldots,i'_{l-1})$ and $i_l<i'_l$. Furthermore, we equip $F_k(\triangle_n)$ and $F_{k+1}(\triangle_n)$ with arbitrary total orders in advance to carry out the following procedure uniquely.
We start with $T_0\coloneqq K(\tau)$, $\varphi_0(\emptyset)\coloneqq\tau$ with $W_0=\{\emptyset\}$, and $\tau_1\coloneqq\tau$. For each step $i\ge1$, we define $M_i\coloneqq\{\sg\in F_{k+1}(X)\mid K(\sg)\cap T_{i-1}=K(\tau_i)\}$. If $m_i\coloneqq\#M_i=0$, then set $T_i\coloneqq T_{i-1}$ and $\varphi_i\coloneqq\varphi_{i-1}$. Otherwise, arrange $M_i=\{\sg_1,\sg_2,\ldots,\sg_{m_i}\}$ in ascending order and define
\[
T_i\coloneqq T_{i-1}\cup\bigcup_{j=1}^{m_i}K(\sg_j). 
\]
Furthermore, for each $1\le j\le m_i$, let $\rho_{j,1},\rho_{j,2},\ldots,\rho_{j,k+1}$ be the ascending order of the $k$-dimensional faces of $\sg_j$, distinct from $\tau_i$. Then, we extend $\varphi_{i-1}$ to $\varphi_i$ such that $\varphi_i^{-1}(\rho_{j,l})\coloneqq(\varphi_{i-1}^{-1}(\tau_i),(k+1)(j-1)+l)$ for $1\le j\le m_i$ and $1\le l\le k+1$. Finally, if $F_k(T_i)\supsetneq\{\tau_1,\tau_2\ldots,\tau_i\}$, then we define
\[
\tau_{i+1}\coloneqq\min(F_k(T_i)\setminus\{\tau_1,\tau_2,\ldots,\tau_i\}), 
\]
where the minimum is taken with respect to the total order in $F_k(T_i)$ induced from $(W_i,<)$ under $\varphi_i$. Otherwise, this process stops, and we encode $I\coloneqq i$. 

We extend the sequences $(m_i)_{i=1}^I$, $(T_i)_{i=1}^I$, and $(\tau_i)_{i=1}^I$ by $m_i\coloneqq m_I=0$, $T_i\coloneqq T_I$, and $\tau_i\coloneqq\tau_I $ for $i>I$, respectively. Furthermore, we define
\[
c_i\coloneqq\begin{cases}
\#\left\{v\in V(T_{i-1})\setminus\tau_i\relmiddle|\tau_i\cup\{v\}\in X\setminus T_{i-1}\right\}	&(1\le i\le I), \\
0																								&(i>I). 
\end{cases}
\]
Clearly, $X(\tau)$ is a $k$-rooted tree if and only if $c_i=0$ for $i\ge1$. We often denote $m_i$, $T_i$, $\tau_i$, $I$, and $c_i$ by $m_i(X,\tau)$, $T_i(X,\tau)$, $\tau_i(X,\tau)$, $I(X,\tau)$, and $c_i(X,\tau)$, respectively, to indicate the explored simplicial complex and the initial $k$-simplex. 

\subsection{Estimates on the breadth-first traversal of random simplicial complexes}\label{ssec:E_BFT}
For $\tau\in F_k(\triangle_n)$ with $\P(\tau\in X_n)>0$, we define a probability space $(\Om_\tau, \cF_\tau, \P_\tau)$ by
\[
\Om_\tau\coloneqq\{\tau\in X_n\}, \,\cF_\tau\coloneqq\{B\in\cF\mid B\subset \Om_\tau \}, \text{ and } \P_\tau(\cdot)\coloneqq\P(\cdot\mid\Om_\tau ). 
\]
The expectation with respect to $\P_\tau$ is denoted by $\E_\tau$. In what follows in this subsection, let $\tau$ be a fixed $k$-simplex in $\triangle_n$ with $\P(\tau\in X_n)>0$. We can carry out the breadth-first traversal of $X_n$ from $\tau$ given $\tau$ appearing in $X_n$, and obtain $m_i(X_n,\tau)$, $T_i(X_n,\tau)$, $\tau_i(X_n,\tau)$, $I(X_n,\tau)$, and $c_i(X_n,\tau)$ that are defined on the probability space $(\Om_\tau,\cF_\tau,\P_\tau)$. For the simplicity, we denote $m_i(X_n,\tau)$, $T_i(X_n,\tau)$, $\tau_i(X_n,\tau)$, $I(X_n,\tau)$, and $c_i(X_n,\tau)$ by $m_i^{(n)}$, $T_i^{(n)}$, $\tau_i^{(n)}$, $I(n)$, and $c_i^{(n)}$, respectively. Moreover, we define a filtration $\cF^{(n)}=\bigl(\cF_i^{(n)}\bigr)_{i=1}^\infty$ on $(\Om_\tau,\cF_\tau,\P_\tau)$ by
\[
\cF_i^{(n)}\coloneqq\sg\bigl(T_j^{(n)};1\le j\le i\bigr). 
\]
It is easy to confirm that $\bigl(m_i^{(n)}\bigr)_{i=1}^\infty$ is $\cF^{(n)}$-adapted and that $I(n)$ is an $\cF^{(n)}$-stopping time. 

To provide some estimates on the breadth-first traversal of $X_n$, we additionally define
\[
s_k\coloneqq\begin{cases}
\P(\tau_1\cup\tau_2\in X_n\mid\tau_1,\tau_2\in X_n)	&\text{if $\P(\tau_1,\tau_2\in X_n)>0$},\\
0													&\text{otherwise.}
\end{cases}
\]
Here, $\tau_1$ and $\tau_2$ are arbitrarily fixed $k$-simplices in $\triangle_n$ such that $\dim(\tau_1\cap\tau_2)=k-1$. This is well-defined from the homogeneity of $X_n$. Whenever $q_k>0$, Eq.~\eqref{eq:D0_SI} implies that $\P(\tau_1,\tau_2\in X_n)=q_k^2/q_{k-1}>0$. Therefore, we have $s_k=q_{k+1}/(q_k^2/q_{k-1})=r_k/r_{k-1}$. 
\begin{lemma}\label{lem:conTre}
Let $T\subset\triangle_n$ be a $(k+1)$-tree with $\P(T\subset X_n)>0$, and let $\sg_1,\sg_2,\ldots,\sg_i\in F_{k+1}(\triangle_n)\setminus T$. Then, $\P(\sg_j\notin X_n\text{ for all }j=1,2,\ldots,i\mid T\subset X_n)\ge1-i s_k$. 
\end{lemma}
\begin{proof}
For each $1\le j\le i$, there exist two distinct $k$-dimensional faces $\tau_j,\tau'_j$ of $\sg_j$ such that $K(\sg_j)\cap T\subset K(\tau_j)\cup K(\tau'_j)$. Therefore, we obtain
\begin{align*}
&\P(\sg_j\in X_n\text{ for some }j=1,2,\ldots,i\mid T\subset X_n)\\
&\le\sum_{j=1}^i\P(\sg_j\in X_n\mid T\subset X_n)\\
&=\sum_{j=1}^i\P(\sg_j\in X_n\mid K(\sg_j)\cap T\subset X_n)\quad\text{(from Eq.~\eqref{eq:D1_SI})}\\
&\le\sum_{j=1}^i\P(\sg_j\in X_n\mid\tau_j,\tau'_j\in X_n)\quad\text{(from Eq.~\eqref{eq:M_CP})}\\
&=i s_k. \qedhere
\end{align*}
\end{proof}
\begin{lemma}\label{lem:growth}
Let $i\ge0$ be fixed, and let $T\subset\triangle_n$ be a $(k+1)$-tree such that $f_k(T)\ge i+1$. Provided that $\P_\tau\bigl(T_i^{(n)}=T\bigr)>0$, for any $v\in[n]$ with $\tau_{i+1}(T,\tau)\cup\{v\}\notin T$, 
\begin{equation}\label{eq:growth1}
\P_\tau\bigl(\tau_{i+1}(T,\tau)\cup\{v\}\in X_n\mid T_i^{(n)}=T\bigr)\le s_k.
\end{equation}
Moreover, if $v\notin V(T)$, then
\begin{equation}\label{eq:growth2}
(1-is_k)r_k\le\P_\tau\bigl(\tau_{i+1}(T,\tau)\cup\{v\}\in X_n\mid T_i^{(n)}=T\bigr)\le r_k. 
\end{equation}
Furthermore, under $\P_\tau\bigl(\cdot\mid T_i^{(n)}=T\bigr)$, the events $(\{\tau_{i+1}(T,\tau)\cup\{w\}\in X_n\})_{w\in[n]\setminus V(T)}$ are mutually independent. 
\end{lemma}
\begin{proof}
Let $v\in[n]$ with $\sg\coloneqq\tau_{i+1}(T,\tau)\cup\{v\}\notin T$.
For $w\in[n]$, we define an event $E_w$ by
\[
E_w\coloneqq\bigcup_{\substack{1\le j\le i\\w\notin V(T_{j-1}(T,\tau))}}\{\tau_j(T,\tau)\cup\{w\}\in X_n\setminus T\}.
\]
Note that
\[
\bigl\{T_i^{(n)}=T\bigr\}=\{T\subset X_n\}\cap\Biggl(\bigcup_{w\in[n]}E_w\Biggr)^c
\]
from the definition of the breadth-first traversal. Therefore, we have
\begin{align}\label{eq:conTre1}
\P_\tau\bigl(\sg\in X_n\mid T_i^{(n)}=T\bigr)
&=\frac{\P_\tau\left(\{\sg\in X_n\}\cap\bigl(\bigcup_{w\in[n]}E_w\bigr)^c\relmiddle|T\subset X_n\right)}{\P_\tau\left(\bigl(\bigcup_{w\in[n]}E_w\bigr)^c\relmiddle|T\subset X_n\right)}\nonumber\\
&=\frac{\P\left(\{\sg\in X_n\}\cap\bigl(\bigcup_{w\in[n]}E_w\bigr)^c\relmiddle|T\subset X_n\right)}{\P\left(\bigl(\bigcup_{w\in[n]}E_w\bigr)^c\relmiddle|T\subset X_n\right)}\nonumber\\
&\le\P(\sg\in X_n\mid T\subset X_n).
\end{align}
In the second line, we use the fact that $\P_\tau(\cdot\mid T\subset X_n)=\P(\cdot\mid T\subset X_n)$. For the last inequality, we use the fact that the events $\{\sg\in X_n\}$ and $\bigcup_{w\in[n]}E_w$ are positively correlated under $\P(\cdot\mid T\subset X_n)$:
\[
\P\left(\{\sg\in X_n\}\cap\bigcup_{w\in[n]}E_w\relmiddle|T\subset X_n\right)
\ge\P(\sg\in X_n\mid T\subset X_n)\P\left(\bigcup_{w\in[n]}E_w\relmiddle|T\subset X_n\right).
\]
This fact follows immediately from the FKG inequality (see, e.g.,~\cite[Theorem~2.4]{Gr}) via Theorem~\ref{thm:cha_HSI}.
Since $\P(\sg\in X_n\mid T\subset X_n)\le s_k$ by Lemma~\ref{lem:conTre}, we obtain Eq.~\eqref{eq:growth1} from Eq.~\eqref{eq:conTre1}.

Next, suppose further that $v\notin V(T)$. Using Eq.~\eqref{eq:D1_SI} with $K(\sg)\cap T=K\bigl(\tau_{i+1}(T,\tau)\bigr)$, we have
\begin{equation}\label{eq:conTre2}
\P(\sg\in X_n\mid T\subset X_n)
=\P(\sg\in X_n\mid\tau_{i+1}(T,\tau)\in X_n)
=r_k.
\end{equation}
Combining this with Eq.~\eqref{eq:conTre1}, we obtain the second inequality in Eq.~\eqref{eq:growth2}.
For the first inequality in Eq.~\eqref{eq:growth2}, we additionally define
\[
E\coloneqq(E_v)^c
=\bigcap_{1\le j\le i}\{\tau_j(T,\tau)\cup\{v\}\notin X_n\}
\quad\text{and}\quad F\coloneqq\bigcap_{w\in[n]\setminus\{v\}}(E_w)^c. 
\]
Note that the event $F$ is independent of both $E$ and $\{\sg\in X_n\}\cap E$ under $\P(\cdot\mid T\subset X_n)$ via Theorem~\ref{thm:cha_HSI}.
Since $\bigl\{T_i^{(n)}=T\bigr\}=\{T\subset X_n\}\cap E\cap F$, we have
\begin{align*}
\P_\tau\bigl(\sg\in X_n\mid T_i^{(n)}=T\bigr)
&=\frac{\P_\tau(\{\sg\in X_n\}\cap E\cap F\mid T\subset X_n)}{\P_\tau(E\cap F\mid T\subset X_n)}\\
&=\frac{\P(\{\sg\in X_n\}\cap E\cap F\mid T\subset X_n)}{\P(E\cap F\mid T\subset X_n)}\\
&=\frac{\P(\{\sg\in X_n\}\cap E\mid T\subset X_n)}{\P(E\mid T\subset X_n)}\\
&\ge\P(\{\sg\in X_n\}\cap E\mid T\subset X_n)\\
&=\P(E\mid\sg\in X_n,T\subset X_n)\P(\sg\in X_n\mid T\subset X_n)\\
&=\P(E\mid\sg\in X_n,T\subset X_n)r_k\quad\text{(from Eq.~\eqref{eq:conTre2})}\\
&\ge(1-is_k)r_k.
\end{align*}
In the last line, we apply Lemma~\ref{lem:conTre} with a $(k+1)$-tree $K(\sg)\cup T$ instead of $T$. This completes Eq.~\eqref{eq:growth2}.

Lastly, we cosider the mutual independence of the events $(\{\tau_{i+1}(T,\tau)\cup\{w\}\in X_n\})_{w\in[n]\setminus V(T)}$. Again note that $\bigl\{T_i^{(n)}=T\bigr\}=\{T\subset X_n\}\cap\bigcap_{x\in[n]}(E_x)^c$. Then, via Theorem~\ref{thm:cha_HSI}, we obtain
\begin{align*}
&\P_\tau\left(\bigcap_{w\in[n]\setminus V(T)}\{\tau_{i+1}(T,\tau)\cup\{w\}\in X_n\}\relmiddle| T_i^{(n)}=T\right)\\
&=\frac{\P\left(\bigcap_{w\in[n]\setminus V(T)}\{\tau_{i+1}(T,\tau)\cup\{w\}\in X_n\}\cap\bigcap_{x\in[n]}(E_x)^c\relmiddle| T\subset X_n\right)}{\P\left(\bigcap_{x\in[n]}(E_x)^c\relmiddle| T\subset X_n\right)}\\
&=\frac{\P\left(\bigcap_{w\in[n]\setminus V(T)}\bigl(\{\tau_{i+1}(T,\tau)\cup\{w\}\in X_n\}\cap(E_w)^c\bigr)\cap\bigcap_{x\in V(T)}(E_x)^c\relmiddle| T\subset X_n\right)}{\P\left(\bigcap_{w\in[n]\setminus V(T)}(E_w)^c\cap\bigcap_{x\in V(T)}(E_x)^c\relmiddle| T\subset X_n\right)}\\
&=\prod_{w\in[n]\setminus V(T)}\frac{\P(\{\tau_{i+1}(T,\tau)\cup\{w\}\in X_n\}\cap(E_w)^c\mid T\subset X_n)}{\P((E_w)^c\mid T\subset X_n)}\\
&=\prod_{w\in[n]\setminus V(T)}\frac{\P\left(\{\tau_{i+1}(T,\tau)\cup\{w\}\in X_n\}\cap\bigcap_{x\in[n]}(E_x)^c\relmiddle| T\subset X_n\right)}{\P\left(\bigcap_{x\in[n]}(E_x)^c\relmiddle| T\subset X_n\right)}\\
&=\prod_{w\in[n]\setminus V(T)}\P_\tau\bigl(\tau_{i+1}(T,\tau)\cup\{w\}\in X_n\mid T_i^{(n)}=T\bigr). 
\end{align*}
This completes the proof. 
\end{proof}

Next, we give some estimates on $m_{i+1}^{(n)}$ and $c_{i+1}^{(n)}$ under $\P_\tau\bigl(\cdot\mid T_i^{(n)}=T\bigr)$. Let $\mu_{i+1}^{(n)}$ denote the distribution of $m_{i+1}^{(n)}$ under $\P_\tau\bigl(\cdot\mid T_i^{(n)}=T\bigr)$.
\begin{prop}\label{prop:TVgrow}
Let $i\ge0$ be fixed, and let $T\subset\triangle_n$ be a $(k+1)$-tree such that $f_k(T)\ge i+1$. Provided that $\P_\tau\bigl(T_i^{(n)}=T\bigr)>0$,
\begin{align*}
\E_\tau\bigl[m_{i+1}^{(n)}\mid T_i^{(n)}=T\bigr]&\le nr_k
\shortintertext{and}
d_\text{TV}\bigl(\mu_{i+1}^{(n)},\po_{(n-k-1)r_k}\bigr)&\le(f_0(T)-k)r_k+inr_ks_k.
\end{align*}
Here, $d_\text{TV}$ indicates the total variation distance between two probability measures on $\Z_{\ge0}$. Furthermore, 
\[
\E_\tau\bigl[c_{i+1}^{(n)}\mid T_i^{(n)}=T\bigr]\le(f_0(T)-k-1)s_k. 
\]
\end{prop}
\begin{proof}
For $v\in[n]\setminus V(T)$, we define $\xi_v\coloneqq1_{\{\tau_{i+1}(T,\tau)\cup\{v\}\in X_n\}}$. From Lemma~\ref{lem:growth} and the homogeneity of $X_n$, the random variables $(\xi_v)_{v\in[n]\setminus V(T)}$ are independent and identically distributed under $\P_\tau\bigl(\cdot\mid T_i^{(n)}=T\bigr)$. Moreover, $p\coloneqq\E_\tau\bigl[\xi_v\mid T_i^{(n)}=T\bigr]\in[(1-is_k)r_k,r_k]$. Therefore, $m_{i+1}^{(n)}$ follows the binomial distribution under $\P_\tau(\cdot\mid T_i^{(n)}=T)$: 
\[
m_{i+1}^{(n)}=\sum_{v\in [n]\setminus V(T)}\xi_v\sim\bin(n-f_0(T),p).
\]
This immediately implies the first conclusion. Furthermore, from some estimates on the total variation distance~(see, e.g.,~\cite[Theorem~1]{BH} and~\cite[Formula~(5)]{R}), we have
\begin{gather*}
d_\text{TV}(\bin(n-f_0(T),p),\po_{(n-f_0(T))p})\le p\le r_k
\shortintertext{and}
d_\text{TV}(\po_{(n-f_0(T))p},\po_{(n-k-1)r_k})\le(n-k-1)r_k-(n-f_0(T))p\le(f_0(T)-k-1)r_k+inr_ks_k.
\end{gather*}
Thus,
using the triangle inequality,
\begin{align*}
&d_\text{TV}\bigl(\mu_{i+1}^{(n)},\po_{(n-k-1)r_k}\bigr)\\
&\le d_\text{TV}(\bin(n-f_0(T),p),\po_{(n-f_0(T))p})+d_\text{TV}(\po_{(n-f_0(T))p},\po_{(n-k-1)r_k})\\
&\le(f_0(T)-k)r_k+inr_ks_k.
\end{align*}

Lastly, again from Lemma~\ref{lem:growth}, we have
\begin{align*}
\E_\tau\bigl[c_{i+1}^{(n)}\mid T_i^{(n)}=T\bigr]
&=\sum_{v\in V(T)\setminus\tau_{i+1}(T,\tau)}\P_\tau\bigl(\tau_{i+1}(T,\tau)\cup\{v\}\in X_n\setminus T\mid T_i^{(n)}=T\bigr)\\
&\le(f_0(T)-k-1)s_k.
\end{align*}
These estimates complete the proof. 
\end{proof}

Let $i\ge0$ be fixed. From Proposition~\ref{prop:TVgrow}, we have
\[
\E_\tau\bigl[m_{i+1}^{(n)}\mid\cF_i^{(n)}\bigr]
=\E_\tau\bigl[m_{i+1}^{(n)}1_{\{f_k(T_i^{(n)})\ge i+1\}}\mid\cF_i^{(n)}\bigr]
=\sum_T\E_\tau\bigl[m_{i+1}^{(n)}\mid T_i^{(n)}=T\bigr]1_{\{T_i^{(n)}=T\}}
\le nr_k.
\]
The summation in the above equation is taken over all $(k+1)$-trees $T$ in $\triangle_n$ such that $f_k(T)\ge i+1$ and $\P_\tau\bigl(T_i^{(n)}=T\bigr)>0$. Therefore, 
\[
\E_\tau\bigl[m_{i+1}^{(n)}\bigr]
=\E_\tau\bigl[\E_\tau\bigl[m_{i+1}^{(n)}\mid\cF_i^{(n)}\bigr]\bigr]
\le nr_k.
\]
Since $f_0\bigl(T_i^{(n)}\bigr)=k+1+\sum_{j=0}^{i-1}m_{j+1}^{(n)}$, we obtain
\begin{equation}\label{eq:vnumber}
\E_\tau\bigl[f_0\bigl(T_i^{(n)}\bigr)\bigr]=k+1+\sum_{j=0}^{i-1}\E_\tau\bigl[m_{j+1}^{(n)}\bigr]\le k+1+inr_k.
\end{equation}
Furthermore, 
\begin{align*}
\E_\tau\bigl[c_{i+1}^{(n)}\mid\cF_i^{(n)}\bigr]
&=\E_\tau\bigl[c_{i+1}^{(n)}1_{\{f_k(T_i^{(n)})\ge i+1\}}\mid\cF_i^{(n)}\bigr]\\
&=\sum_T\E_\tau\bigl[c_{i+1}^{(n)}\mid T_i^{(n)}=T\bigr]1_{\{T_i^{(n)}=T\}}\\
&\le\sum_T(f_0(T)-k-1)s_k1_{\{T_i^{(n)}=T\}}\quad\text{(from Proposition~\ref{prop:TVgrow})}\\
&\le\bigl(f_0\bigl(T_i^{(n)}\bigr)-k-1\bigr)s_k.
\end{align*}
The summations in the second and third lines are also taken over all $(k+1)$-trees $T$ in $\triangle_n$ such that $f_k(T)\ge i+1$ and $\P\bigl(T_i^{(n)}=T\bigr)>0$. Therefore,
from Eq.~\eqref{eq:vnumber}, we have
\begin{equation}\label{eq:treObs}
\E_\tau\bigl[c_{i+1}^{(n)}\bigr]
=\E_\tau\bigl[\E_\tau\bigl[c_{i+1}^{(n)}\mid\cF_i^{(n)}\bigr]\bigr]
\le\bigl(\E_\tau\bigl[f_0\bigl(T_i^{(n)}\bigr)\bigr]-k-1\bigr)s_k
\le inr_ks_k. 
\end{equation}

Now, we define a nondecreasing sequence $(I(n;l))_{l=0}^\infty$ of $\cF^{(n)}$-stopping times by
\[
I(n;l)\coloneqq\begin{cases}
0																												&(l=0),\\
\sup\{1\le i\le I(n)\mid\text{the distance of $\tau_i^{(n)}$ from $\tau$ is less than $l$}\}	&(l\ge1).
\end{cases} 
\]
Let us denote $c_n=(n-k-1)r_k$ and set $f_k(\a)\coloneqq f_k(T)$ for any $k$-rooted tree $\a=(T,\tau)$. 
\begin{lemma}\label{lem:probdiff}
Let $l\in\N$ be fixed, and let $\a$ be a $k$-rooted tree.
Then,
\begin{align*}
\bigl|\P_\tau\bigl(\bigl(T_{I(n;l)}^{(n)},\tau\bigr)\simeq\a_l\bigr)-\P'((\pt_k(c_n),\tau_o)_l\simeq\a_l)\bigr|\le\{1+f_k(\a_l)n(r_k+s_k)/2\}f_k(\a_l)r_k.
\end{align*}
\end{lemma}
\begin{proof}
For each $i\ge0$, we define a new random $k$-rooted tree $\bigl(\tilde T_i^{(n)},\tau\bigr)$ as follows: we start with $\bigl(T_i^{(n)},\tau\bigr)$; and for each $j=i+1,i+2,\ldots,f_k\bigl(T_i^{(n)}\bigr)$, we attach a $k$-rooted Poisson tree with parameter $c_n$ as a branch rooted at $\tau_j^{(n)}$. Let $Z_j$ denote the degree of the root in the $k$-rooted Poisson tree
attached to $\tau_j^{(n)}$.
We may assume that each $k$-rooted Poisson tree is defined on the same probability space $(\Om,\cF,\P)$ to be independent of $X_n$ and the other trees. Note that $\bigl(\tilde T_0^{(n)},\tau\bigr)$ is a $k$-rooted Poisson tree with parameter $c_n$. Now, we define $I(h)\coloneqq f_k(\a_{h-1})$ for $h\ge1$, and $I(0)\coloneqq0$. Then,
\begin{align}\label{eq:probdiff1}
&\bigl|\P_\tau\bigl(\bigl(T_{I(n;l)}^{(n)},\tau\bigr)\simeq\a_l\bigr)-\P'((\pt_k(c_n),\tau_o)_l\simeq\a_l)\bigr|\nonumber\\
&=\bigl|\P_\tau\bigl(\bigl(\tilde T_{I(l)}^{(n)},\tau\bigr)_l\simeq\a_l\bigr)-\P_\tau\bigl(\bigl(\tilde T_0^{(n)},\tau\bigr)_l\simeq\a_l\bigr)\bigr|\nonumber\\
&\le\sum_{h=0}^{l-1}\sum_{i=I(h)}^{I(h+1)-1}\bigl|\P_\tau\bigl(\bigl(\tilde T_{i+1}^{(n)},\tau\bigr)_l\simeq\a_l\bigr)-\P_\tau\bigl(\bigl(\tilde T_i^{(n)},\tau\bigr)_l\simeq\a_l\bigr)\bigr|\nonumber\\
&\le\sum_{h=0}^{l-1}\sum_{i=I(h)}^{I(h+1)-1}\bigl|\P_\tau\bigl(\bigl(\tilde T_{i+1}^{(n)},\tau\bigr)_{h+1}\simeq\a_{h+1}\bigr)-\P_\tau\bigl(\bigl(\tilde T_i^{(n)},\tau\bigl)_{h+1}\simeq\a_{h+1}\bigr)\bigr|. 
\end{align}
In the last line above, we use the fact that
\[
\P_\tau\bigl(\bigl(\tilde T_{i+1}^{(n)},\tau\bigr)_l\simeq\a_l\mid\bigl(\tilde T_{i+1}^{(n)},\tau\bigr)_{h+1}\simeq\a_{h+1}\bigr)
=\P_\tau\bigl(\bigl(\tilde T_i^{(n)},\tau\bigr)_l\simeq\a_l\mid\bigl(\tilde T_i^{(n)},\tau\bigr)_{h+1}\simeq\a_{h+1}\bigr). 
\]

Now, we estimate the summand in the last line of Eq.~\eqref{eq:probdiff1} for a fixed $0\le h\le l-1$ and $I(h)\le i<I(h+1)$. Let $\cT_i^{(n)}$ be the set of all $(k+1)$-trees $T$ in $\triangle_n$ such that $\P_\tau\bigl(T_i^{(n)}=T\bigr)>0$. 
For $T\in\cT_i^{(n)}$
with $(T,\tau)_h\simeq\a_h$
and $(m_1,\ldots,m_{I(h+1)-i})\in\Z_{\ge0}^{I(h+1)-i}$, define a new $(k+1)$-tree $T(m_1,\ldots,m_{I(h+1)-i})$ as follows: for each $j=1,2,\ldots,I(h+1)-i$, we pick $m_j$ numbers of new vertices $v_1^{(j)},\ldots,v_{m_j}^{(j)}$ and add $(k+1)$-simplices $\tau_{i+j}(T,\tau)\cup\bigl\{v_1^{(j)}\bigr\},\ldots,\tau_{i+j}(T,\tau)\cup\bigl\{v_{m_j}^{(j)}\bigr\}$ to $T$. Furthermore, we define
\[
\cM(T,\a_{h+1})\coloneqq\left\{(m_1,\ldots,m_{I(h+1)-i})\in\Z_{\ge0}^{I(h+1)-i}\relmiddle|(T(m_1,\ldots,m_{I(h+1)-i}),\tau)\simeq\a_{h+1}\right\}
\]
for $T\in\cT_i^{(n)}$
with $(T,\tau)_h\simeq\a_h$.
Then, we have
\begin{align}
&\P_\tau\bigl(\bigl(\tilde T_{i+1}^{(n)},\tau\bigr)_{h+1}\simeq\a_{h+1}\bigr)\nonumber\\
&=\sum_{\substack{T\in\cT_i^{(n)}\\(T,\tau)_h\simeq\a_h}}\P_\tau\bigl(\bigl(m_{i+1}^{(n)},Z_{i+2},\ldots,Z_{I(h+1)}\bigr)\in\cM(T,\a_{h+1})\mid T_i^{(n)}=T\bigr)\P_\tau\bigl(T_i^{(n)}=T\bigr)\label{eq:probdiff2}
\shortintertext{and}
&\P_\tau\bigl(\bigl(\tilde T_i^{(n)},\tau\bigr)_{h+1}\simeq\a_{h+1}\bigr)\nonumber\\
&=\sum_{\substack{T\in\cT_i^{(n)}\\(T,\tau)_h\simeq\a_h}}\P_\tau\bigl((Z_{i+1},Z_{i+2},\ldots,Z_{I(h+1)})\in\cM(T,\a_{h+1})\mid T_i^{(n)}=T\bigr)\P_\tau\bigl(T_i^{(n)}=T\bigr)\nonumber\\
&=\sum_{\substack{T\in\cT_i^{(n)}\\(T,\tau)_h\simeq\a_h}}\po_{c_n}^{\otimes I(h+1)-i}(\cM(T,\a_{h+1}))\P_\tau\bigl(T_i^{(n)}=T\bigr). \label{eq:probdiff3}
\end{align}
Furthermore, a simple calculation implies that
\begin{align}\label{eq:probdiff4}
&\bigl|\P_\tau\bigl(\bigl(m_{i+1}^{(n)},Z_{i+2},\ldots,Z_{I(h+1)}\bigr)\in\cM(T,\a_{h+1})\mid T_i^{(n)}=T\bigr)-\po_{c_n}^{\otimes I(h+1)-i}(\cM(T,\a_{h+1}))\bigr|\nonumber\\
&\le\sup_{A\subset\Z_{\ge0}^{I(h+1)-i}}\bigl|\P_\tau\bigl(\bigl(m_{i+1}^{(n)},Z_{i+2},\ldots,Z_{I(h+1)}\bigr)\in A\mid T_i^{(n)}=T\bigr)-\po_{c_n}^{\otimes I(h+1)-i}(A)\bigr|\nonumber\\
&\le\sup_{A\subset\Z_{\ge0}}\bigl|\P_\tau\bigl(m_{i+1}^{(n)}\in A\mid T_i^{(n)}=T\bigr)-\po_{c_n}(A)\bigr|\nonumber\\
&\le\bigl(f_0(T)-k\bigr)r_k+inr_ks_k\quad\text{(from Proposition~\ref{prop:TVgrow}).}
\end{align}
Combining Eqs.~\eqref{eq:probdiff2},~\eqref{eq:probdiff3}, and~\eqref{eq:probdiff4}, we have
\begin{align}\label{eq:probdiff5}
&\bigl|\P_\tau\bigl(\bigl(\tilde T_{i+1}^{(n)},\tau\bigr)_{h+1}\simeq\a_{h+1}\bigr)-\P_\tau\bigl(\bigl(\tilde T_i^{(n)},\tau\bigr)_{h+1}\simeq\a_{h+1}\bigr)\bigr|\nonumber\\
&\le\sum_{\substack{T\in\cT_i^{(n)}\\(T,\tau)_h\simeq\a_h}}\bigl\{(f_0(T)-k)r_k+inr_ks_k\bigr\}\P_\tau\bigl(T_i^{(n)}=T\bigr)\nonumber\\
&\le\bigl(\E_\tau\bigl[f_0\bigl(T_i^{(n)}\bigr)\bigr]-k\bigr)r_k+inr_ks_k\nonumber\\
&\le\{1+in(r_k+s_k)\}r_k\quad\text{(from Eq.~\eqref{eq:vnumber}).}
\end{align}
Thus, from Eqs.~\eqref{eq:probdiff1} and~\eqref{eq:probdiff5}, we obtain
\begin{align*}
\bigl|\P_\tau\bigl(\bigl(T_{I(n;l)}^{(n)},\tau\bigr)\simeq\a_l\bigr)-\P'((\pt_k(c_n),\tau_o)_l\simeq\a_l)\bigr|
&\le\sum_{h=0}^{l-1}\sum_{i=I(h)}^{I(h+1)-1}\{1+in(r_k+s_k)\}r_k\\
&=\sum_{i=0}^{I(l)-1}\{1+in(r_k+s_k)\}r_k\\
&\le\{1+I(l)n(r_k+s_k)/2\}I(l)r_k.
\end{align*}
This completes the proof. 
\end{proof}
\begin{prop}\label{prop:mass_dif}
Let $l\in\N$ be fixed, and let $\a$ be a $k$-rooted tree.
Then,
\begin{align*}
\bigl|\P_\tau((X_n,\tau)_l\simeq\a_l)-\P'((\pt_k(c_n),\tau_o)_l\simeq\a_l)\bigr|
\le\{1+f_k(\a_l)n(r_k/2+s_k)\}f_k(\a_l)r_k.
\end{align*}
\end{prop}
\begin{proof}
We define $Q_n\coloneqq\Om_\tau\cap\bigl\{c_i^{(n)}=0\text{ for all }i=1,2,\ldots,I(n;l)\bigr\}$. Note that $(X_n,\tau)_l=\bigl(T_{I(n;l)}^{(n)},\tau\bigr)$ given the event $Q_n$. Thus, we have
\begin{align*}
&\bigl|\P_\tau((X_n,\tau)_l\simeq\a_l)-\P_\tau\bigl(\bigl(T_{I(n;l)}^{(n)},\tau\bigr)\simeq\a_l\bigr)\bigr|\\
&\le\P_\tau(\{(X_n,\tau)_l\simeq\a_l\}\setminus Q_n)\vee\P_\tau\bigl(\bigl\{\bigl(T_{I(n;l)}^{(n)},\tau\bigr)\simeq\a_l\bigr\}\setminus Q_n\bigr)\\
&\le\P_\tau(\{I(n;l)\le f_k(\a_l)\}\setminus Q_n)\\
&\le\P_\tau\Biggl(\sum_{i=0}^{f_k(\a_l)-1}c_{i+1}^{(n)}\ge1\Biggr)\\
&\le\sum_{i=0}^{f_k(\a_l)-1}\E_\tau\bigl[c_{i+1}^{(n)}\bigr]\quad\text{(from Markov's inequality)}\\
&\le\sum_{i=0}^{f_k(\a_l)-1}inr_ks_k\quad\text{(from Eq.~\eqref{eq:treObs})}\\
&\le f_k(\a_l)^2nr_ks_k/2.
\end{align*}

Combining this estimate with Lemma~\ref{lem:probdiff}, we obtain
\begin{align*}
&|\P_\tau((X_n,\tau)_l\simeq\a_l)-\P'((\pt_k(c_n),\tau_o)_l\simeq\a_l)|\\
&\le\bigl|\P_\tau((X_n,\tau)_l\simeq\a_l)-\P_\tau\bigl(\bigl(T_{I(n;l)}^{(n)},\tau\bigr)\simeq\a_l\bigr)\bigr|
+\bigl|\P_\tau\bigl(\bigl(T_{I(n;l)}^{(n)},\tau\bigr)\simeq\a_l\bigr)-\P'((\pt_k(c_n),\tau_o)_l\simeq\a_l)\bigr|\\
&\le f_k(\a_l)^2nr_ks_k/2+\{1+f_k(\a_l)n(r_k+s_k)/2\}f_k(\a_l)r_k\\
&=\{1+f_k(\a_l)n(r_k/2+s_k)\}f_k(\a_l)r_k.\qedhere
\end{align*}
\end{proof}

\subsection{Proof of Theorem~\ref{thm:LWLT_HSIRSC}}
The following lemma gives some fundamental relations among $q_k$, $r_k$, and $s_k$.
\begin{lemma}\label{lem:qrs}
The following $(1)$, $(2)$, and $(3)$ hold.
\begin{enumerate}
\item $q_i^{k+1}\ge q_k^{i+1}$ for $0\le i\le k$. In particular, $n^{k+1}q_k=\om(1)$ implies that $n^{i+1}q_i=\om(1)$. 
\item $r_k\ge q_0s_k^{k+1}$. In particular, $nq_0=\om(1)\text{ and }nr_k\asymp1$ together imply that $s_k=o(1)$. 
\item $q_ks_k^{k+1}\ge r_k^{k+1}$. In particular, $s_k=o(1)\text{ and }nr_k\asymp1$ together imply that $n^{k+1}q_k=\om(1)$. 
\end{enumerate}
Furthermore, if $nr_k\asymp1$, then the following three conditions are equivalent$:$ $n^{k+1}q_k=\om(1);$ $nq_0=\om(1);$ $s_k=o(1)$. 
\end{lemma}
\begin{proof}
When $q_k=0$, the conclusions are trivial because $q_k=0$ implies that $r_k=s_k=0$. Hence, we may assume $q_0\ge q_1\ge\cdots\ge q_k>0$. Since
\[
\frac{(n^{i+2}q_{i+1})(n^iq_{i-1})}{(n^{i+1}q_i)^2}=\frac{r_i}{r_{i-1}}=s_i\le1\quad\text{for }0\le i\le k,
\]
the sequence $(n^{i+1}q_i)_{-1\le i\le k}$ is log-concave. Therefore, noting that $q_{-1}=1$, we have $n^{i+1}q_i\ge(n^{k+1}q_k)^{(i+1)/(k+1)}$ for $0\le i\le k$, which completes~(1). Furthermore, $r_{-1}=q_0$ and $r_k$ is nonincreasing with respect to $k$, so
\[
q_k=q_{k-1}r_{k-1}=q_{k-2}r_{k-2}r_{k-1}=\cdots=q_0r_0\cdots r_{k-2}r_{k-1}\ge(r_{k-1})^{k+1}. 
\]
Thus, we obtain $q_ks_k^{k+1}\ge(r_{k-1}s_k)^{k+1}=r_k^{k+1}$, which corresponds with~(3). For~(2), we additionally define
\[
t_i=\P(\tau_1\cup\tau_2\cup\tau_3\in X_n\mid\tau_1,\tau_2,\tau_3\in X_n)\quad\text{for }1\le i\le k.
\]
Here, $\tau_1$, $\tau_2$, and $\tau_3$ are arbitrarily fixed $i$-simplices in $\triangle_n$ such that $\dim(\tau_1\cap\tau_2)=\dim(\tau_2\cap\tau_3)=\dim(\tau_3\cap\tau_1)=i-1$ and $\dim(\tau_1\cap\tau_2\cap\tau_3)=i-2$. Then, Eq.~\eqref{eq:D0_SI} implies that $\P(\tau_1,\tau_2,\tau_3\in X_n)=q_i^3q_{i-2}/q_{i-1}^3$ for $1\le i\le k$. Therefore, we have $t_i=q_{i+1}/(q_i^3q_{i-2}/q_{i-1}^3)=r_ir_{i-2}/r_{i-1}^2=s_i/s_{i-1}$, which implies that $s_i$ is nonincreasing with respect to $i$. Thus, 
\[
r_k=r_{k-1}s_k=r_{k-2}s_{k-1}s_k=\cdots=q_0s_0s_1\cdots s_k\ge q_0s_k^{k+1},
\]
which corresponds with (2). The last conclusion of the theorem follows immediately from (1), (2), and (3). 
\end{proof}
\begin{prop}\label{prop:locConv}
Let $c>0$ and $\tau\in F_k(\triangle_n)$ be fixed. If $n^{k+1}q_k=\om(1)$ and $nr_k\sim c$, then $X_n[\tau]$ under $\P_\tau$ converges to $[\pt_k(c),\tau_o]$ in distribution as $n\to\infty$. 
\end{prop}
\begin{proof}
Let $[\a]\in\cS_k$ and $l\in\N$. From Lemma~\ref{lem:chara_LWC}, it suffices to prove that
\[
\lim_{n\to\infty}\P_\tau((X_n,\tau)_l\simeq\a_l)=\P'((\pt_k(c),\tau_o)_l\simeq\a_l). 
\]
From Proposition~\ref{prop:mass_dif}, we have
\begin{align*}
\bigl|\P_\tau((X_n,\tau)_l\simeq\a_l)-\P'((\pt_k(c_n),\tau_o)_l\simeq\a_l)\bigr|
\le\{1+f_k(\a_l)n(r_k/2+s_k)\}f_k(\a_l)r_k.
\end{align*}
Since $nr_k\sim c$ and $s_k=o(1)$
from Lemma~\ref{lem:qrs},
the right-hand side converges to zero as $n\to\infty$. Furthermore, noting that $\lim_{n\to\infty}c_n=c$, we have
\[
\lim_{n\to\infty}\P'((\pt_k(c_n),\tau_o)_l\simeq\a_l)=\P'((\pt_k(c),\tau_o)_l\simeq\a_l).
\]
These estimates complete the proof. 
\end{proof}

We can also prove the two-root version of Proposition~\ref{prop:locConv} in the same manner. To state the proposition, for $\tau\neq\tau'\in F_k(\triangle_n)$ such that $\P(\tau,\tau'\in X)>0$, we define a probability space $(\Om_{\tau,\tau'}, \cF_{\tau,\tau'}, \P_{\tau,\tau'})$ by
\[
\Om_{\tau,\tau'}\coloneqq\{\tau,\tau'\in X\}, \,\cF_{\tau,\tau'}\coloneqq\{B\in\cF\mid B\subset \Om_{\tau,\tau'} \}, \text{ and } \P_{\tau,\tau'}(\cdot)\coloneqq\P(\cdot\mid\Om_{\tau,\tau'} ). 
\]
The expectation with respect to $\P_{\tau,\tau'}$ is denoted by $\E_{\tau,\tau'}$. For two disjoint $k$-simplices $\tau$, $\tau'$ in $\triangle_n$, we carry out each breadth-first traversal of $X_n$ from $\tau$ and $\tau'$ alternately, avoiding each other. We can carefully modify the estimates in Section~\ref{ssec:E_BFT} for the two-root version of breadth-first traversal, and we can confirm that for any $[\a],[\b]\in\cS_k$ and $l,m\in\N$, 
\[
\lim_{n\to\infty}\P_{\tau,\tau'}((X_n,\tau)_l\simeq\a_l,(X_n,\tau')_m\simeq\b_m)=\P'((\pt_k(c),\tau_o)_l\simeq\a_l)\P'((\pt_k(c),\tau_o)_m\simeq\b_m). 
\]
The above equation yields the following proposition.
\begin{prop}\label{prop:locConv1}
Let $c>0$, and let $\tau,\tau'\in F_k(\triangle_n)$ be fixed to be disjoint. If $n^{k+1}q_k=\om(1)$ and $nr_k\sim c$, then $(X_n[\tau],X_n[\tau'])$ under $\P_{\tau,\tau'}$ converges to $([\pt_k(c),\tau_o],[\pt'_k(c),\tau'_o])$ in distribution as $n\to\infty$. Here, $[\pt'_k(c),\tau'_o]$ is an independent copy of $[\pt_k(c),\tau_o]$. 
\end{prop}

We now turn to proving Theorem~\ref{thm:LWLT_HSIRSC}. The following lemma states that the number of simplicies in $X_n$ is concentrated around its mean.
\begin{lemma}\label{lem:num_faces}
Provided that $n^{k+1}q_k=\om(1)$, it holds that for any $r\in[1,\infty)$, 
\[
\lim_{n\to\infty}\E\biggl[\biggl|\frac{f_k(X_n)}{n^{k+1}q_k}-\frac1{(k+1)!}\biggl|^r\biggr]=0. 
\]
In particular, $\lim_{n\to\infty}\P(f_k(X_n)>0)=1$. 
\end{lemma}
\begin{proof}
It suffices to prove that for all $m\in\N$, 
\[
\lim_{n\to\infty}\E\left[\left(\frac{f_k(X_n)}{n^{k+1}q_k}\right)^m\right]=\biggl(\frac1{(k+1)!}\biggr)^m. 
\]
Indeed, if it holds, then for any $r\in[1,\infty)$, 
\begin{align*}
\E\biggl[\biggl|\frac{f_k(X_n)}{n^{k+1}q_k}-\frac1{(k+1)!}\biggr|^r\biggr]^{2\lceil r\rceil/r}
&\le\E\biggl[\biggl(\frac{f_k(X_n)}{n^{k+1}q_k}-\frac1{(k+1)!}\biggr)^{2\lceil r\rceil}\biggr]\quad\text{(from H\"older's inequality)}\\
&=\sum_{m=0}^{2\lceil r\rceil}\binom{2\lceil r\rceil}m\E\left[\left(\frac{f_k(X_n)}{n^{k+1}q_k}\right)^m\right]\left(\frac{-1}{(k+1)!}\right)^{2\lceil r\rceil-m}\\
&\xrightarrow[n\to\infty]{}0. 
\end{align*}

Now, for $m\in\N$ and
$1\le l\le m$,
define
\[
x_{m,l}\coloneqq\sum_{\tau_1,\ldots,\tau_m}\P(\tau_1,\tau_2,\ldots,\tau_m\in X_n), 
\]
where the summation is taken over all $\tau_1,\ldots,\tau_m\in F_k(\triangle_n)$ such that for each $h=l+1,\ldots,m$, the simplex $\tau_h$ is disjoint from the others. Clearly, $x_{m,l}$ is nondecreasing with respect to $l$. Furthermore, we have
\begin{align*}
\E\left[\left(\frac{f_k(X_n)}{n^{k+1}q_k}\right)^m\right]
&=\frac1{(n^{k+1}q_k)^m}\E\Biggl[\Biggl(\sum_{\tau\in F_k(\triangle_n)}1_{\{\tau\in X_n\}}\Biggr)^m\Biggr]\\
&=\frac1{(n^{k+1}q_k)^m}\sum_{\tau_1,\ldots,\tau_m\in F_k(\triangle_n)}\P(\tau_1,\tau_2,\ldots,\tau_m\in X_n)\\
&=\frac{x_{m,m}}{(n^{k+1}q_k)^m}. 
\end{align*}
Hence, we prove that for all $m\in\N$, 
\begin{equation}\label{eq:mpower1}
\lim_{n\to\infty}\frac{x_{m,m}}{(n^{k+1}q_k)^m}=\biggl(\frac1{(k+1)!}\biggr)^m. 
\end{equation}
We use an inductive argument on $m\in\N$. When $m=1$, the conclusion is trivial. Assume that Eq.~\eqref{eq:mpower1} holds up to $m-1$ for some $m\ge2$. For each
$1\le l\le m-1$,
we have
\begin{align*}
&x_{m,l+1}-x_{m,l}\\
&=\sum_{i=1}^{k+1}\sum_{\tau_1,\ldots,\tau_m}\P(\tau_1,\tau_2,\ldots,\tau_m\in X_n)\\
&=\sum_{i=1}^{k+1}\sum_{\tau_1,\ldots,\tau_m}\P(\tau_1,\tau_2,\ldots,\tau_l\in X_n)\P(\tau_{l+1}\in X_n\mid\tau_1,\tau_2,\ldots,\tau_l\in X_n)q_k^{m-l-1}\\
&\le\sum_{i=1}^{k+1}\sum_{\tau_1,\ldots,\tau_m}\P(\tau_1,\tau_2,\ldots,\tau_l\in X_n)(q_k/q_{i-1})q_k^{m-l-1}\\
&\le\sum_{i=1}^{k+1}\binom{(k+1)l}i\binom n{k+1-i}\binom n{k+1}^{m-l-1}\sum_{\tau_1,\ldots,\tau_l\in F_k(\triangle_n)}\P(\tau_1,\tau_2,\ldots,\tau_l\in X_n)q_k^{m-l}/q_{i-1}\\
&\le((k+1)l)^{k+1}(n^{k+1}q_k)^{m-l}x_{l,l}\sum_{i=1}^{k+1}\frac1{n^iq_{i-1}}.
\end{align*}
Here, the summations in the second, third, and fourth lines are taken over all $\tau_1,\ldots,\tau_m\in F_k(\triangle_n)$ such that for each $h=l+2,\ldots,m$, the simplex $\tau_h$ is disjoint from the others and
\[
\#\bigl((\tau_1\cup\cdots\cup\tau_l)\cap\tau_{l+1}\bigr)=i.
\]
Therefore, by the assumption of induction and Lemma~\ref{lem:qrs}~(1),
\[
\frac{x_{m,l+1}-x_{m,l}}{(n^{k+1}q_k)^m}\le\frac{((k+1)l)^{k+1}x_{l,l}}{(n^{k+1}q_k)^l}\sum_{i=1}^{k+1}\frac1{n^iq_{i-1}}=o(1). 
\]
Furthermore, 
\[
\lim_{n\to\infty}\frac{x_{m,1}}{(n^{k+1}q_k)^m}=\lim_{n\to\infty}\frac{n!}{((k+1)!)^m(n-(k+1)m)!n^{(k+1)m}}=\biggl(\frac1{(k+1)!}\biggr)^m. 
\]
Thus, we obtain
\[
\lim_{n\to\infty}\frac{x_{m,m}}{(n^{k+1}q_k)^m}=\lim_{n\to\infty}\sum_{l=1}^{m-1}\frac{x_{m,l+1}-x_{m,l}}{(n^{k+1}q_k)^m}+\frac{x_{m,1}}{(n^{k+1}q_k)^m}=\biggl(\frac1{(k+1)!}\biggr)^m. \qedhere
\]
\end{proof}

Finally, we move on to proving Theorem~\ref{thm:LWLT_HSIRSC}. Let us denote $\nu g=\int_{\cS_k}g\,d\nu$ for $g\in C_b(\cS_k)$ and $\nu\in\cP_{\cS_k}$ by convention. 
\begin{proof}[Proof of Theorem~\ref{thm:LWLT_HSIRSC}]
It suffices to prove that for any $g\in C_b(\cS_k)$, 
\begin{align}
\lim_{n\to\infty}\E[\lm_k(X_n)g\mid\dim X_n\ge k]&=\nu_k(c)g\label{eq:1moment}
\shortintertext{and}
\lim_{n\to\infty}\E\bigl[(\lm_k(X_n)g)^2\mid\dim X_n\ge k\bigr]&=(\nu_k(c)g)^2. \label{eq:2moment}
\end{align}
Indeed, Eqs.~\eqref{eq:1moment} and~\eqref{eq:2moment} together imply that
\[
\lim_{n\to\infty}\E\bigl[(\lm_k(X_n)g-\nu_k(c)g)^2\mid\dim X_n\ge k\bigr]=0. 
\]
In particular, $\lm_k(X_n)g$ under $P(\cdot\mid\dim X_n\ge k)$ converges to $\nu_k(c)g$ in distribution as $n\to\infty$ for any $g\in C_b(\cS_k)$, which is equivalent to the conclusion~(see, e.g.,~\cite[Theorem~4.11]{Kal}). Hence, suppose that $g\in C_b(\cS_k)$. We define a finite measure on $\cS_k$ by
\[
\tilde\lm_k(X_n)\coloneqq\frac1{\binom n{k+1}q_k}\sum_{\rho\in F_k(X_n)}\dl_{X[\rho]}. 
\]

We begin by considering Eq.~\eqref{eq:1moment}. Let $\tau\in F_k(\triangle_n)$ be arbitrarily fixed. Then, for any $A\in\cB_{\cS_k}$, we have
\begin{align*}
\bigl(\E\tilde\lm_k(X_n)\bigr)(A)
&=\frac1{\binom n{k+1}q_k}\sum_{\rho\in F_k(\triangle_n)}\P(\rho\in X_n,X_n[\rho]\in A)\\
&=\P(\tau\in X_n,X_n[\tau]\in A)/q_k	\quad\text{(from the homogeneity of $X_n$)}\\
&=\P_\tau(X_n[\tau]\in A),
\end{align*}
which implies that $\E\bigl[\tilde\lm_k(X_n)g\bigr]=\bigl(\E\tilde\lm_k(X_n)\bigr)g=\E_{\tau}[g(X_n[\tau])]$. Therefore, Proposition~\ref{prop:locConv} yields
\begin{equation}\label{eq:siki1}
\lim_{n\to\infty}\E\bigl[\tilde\lm_k(X_n)g\bigr]=\E'[g([\pt_k(c),\tau_o])]=\nu_k(c)g. 
\end{equation}
Furthermore, we have
\begin{align*}
&\bigl|\E[\lm_k(X_n)g\mid\dim X_n\ge k]-\E\bigl[\tilde\lm_k(X_n)g\bigr]\bigr|\\
&=\Biggl|\E\Biggl[\biggl(\frac{1_{\{f_k(X_n)>0\}}}{f_k(X_n)\P(f_k(X_n)>0)}-\frac1{\binom n{k+1}q_k}\biggr)\sum_{\rho\in F_k(X_n)}g(X_n[\rho])\Biggr]\Biggr|\\
&\le\|g\|_\infty\E\biggl[\biggl|\frac{1_{\{f_k(X_n)>0\}}}{\P(f_k(X_n)>0)}-\frac{f_k(X_n)}{\binom n{k+1}q_k}\biggr|\biggr]\\
&\le\|g\|_\infty\biggl(\E\biggl[\biggl|\frac{1_{\{f_k(X_n)>0\}}}{\P(f_k(X_n)>0)}-1\biggr|\biggr]+\E\biggl[\biggl|\frac{f_k(X_n)}{\binom n{k+1}q_k}-1\biggr|\biggr]\biggr). 
\end{align*}
Here, $\|g\|_\infty$ indicates the supremum norm of $g$. From Lemma~\ref{lem:num_faces}, the last line above converges to zero as $n\to\infty$. Thus, combining this estimate with Eq.~\eqref{eq:siki1} yields Eq.~\eqref{eq:1moment}. 

Next, we consider Eq.~\eqref{eq:2moment}. Let $\tau,\tau'\in F_k(\triangle_n)$ be arbitrarily fixed to be disjoint. From the homogeneity of $X_n$, we have
\begin{align}
&\E\bigl[\bigl(\tilde\lm_k(X_n)g\bigr)^2\bigr]\nonumber\\
&=\frac1{\binom n{k+1}^2q_k^2}\E\Biggl[\Biggl(\sum_{\rho\in F_k(\triangle_n)}1_{\{\rho\in X_n\}}g(X_n[\rho])\Biggr)^2\Biggr]\nonumber\\
&=\frac1{\binom n{k+1}^2q_k^2}\sum_{i=0}^{k+1}\sum_{\substack{\rho,\rho'\in F_k(\triangle_n)\\\#(\rho\cap\rho')=i}}\E[1_{\{\rho,\rho'\in X_n\}}g(X_n[\rho])g(X_n[\rho'])]\nonumber\\
&=\frac{\binom{n-k-1}{k+1}}{\binom n{k+1}}\E_{\tau,\tau'}[g(X_n[\tau])g(X_n[\tau'])]+\sum_{i=1}^{k+1}\frac1{\binom n{k+1}^2q_{i-1}}\sum_{\substack{\rho,\rho'\in F_k(\triangle_n)\\\#(\rho\cap\rho')=i}}\E_{\rho,\rho'}[g(X_n[\rho])g(X_n[\rho'])]. \label{eq:siki3}
\end{align}
Proposition~\ref{prop:locConv1} implies that
\[
\lim_{n\to\infty}\E_{\tau,\tau'}[g(X_n[\tau])g(X_n[\tau'])]=\E'[g([\pt_k(c),\tau_o])]^2=(\nu_k(c)g)^2. 
\]
Therefore, the first term of Eq.~\eqref{eq:siki3} converges to $(\nu_k(c)g)^2$ as $n\to\infty$. Furthermore, for each $1\le i\le k+1$, 
\[
\frac1{\binom n{k+1}^2q_{i-1}}\left|\sum_{\substack{\rho,\rho'\in F_k(\triangle_n)\\\#(\rho\cap\rho)=i}}\E_{\rho,\rho'}[g(X_n[\rho])g(X_n[\rho'])]\right|
\le\frac{\binom{k+1}i\binom{n-k-1}{k+1-i}\|g\|_\infty^2}{\binom n{k+1}q_{i-1}}
\le\frac{\binom{k+1}in^{k+1}\|g\|_\infty^2}{\binom n{k+1}n^iq_{i-1}}. 
\]
From Lemma~\ref{lem:qrs}~(1), the right-hand side of the above equation converges to zero as $n\to\infty$. Thus, we obtain
\begin{equation}\label{eq:siki4}
\lim_{n\to\infty}\E\bigl[\bigl(\tilde\lm_k(X_n)g\bigr)^2\bigr]=(\nu_k(c)g)^2. 
\end{equation}
We also have
\begin{align*}
&\bigl|\E\bigl[(\lm_k(X_n)g)^2\mid\dim X_n\ge k\bigr]-\E\bigl[\bigl(\tilde\lm_k(X_n)g\bigr)^2\bigr]\bigr|\\
&=\Biggl|\E\Biggl[\biggl(\frac{1_{\{f_k(X_n)>0\}}}{f_k(X_n)^2\P(f_k(X_n)>0)}-\frac1{\binom n{k+1}^2q_k^2}\biggr)\Biggl(\sum_{\rho\in F_k(X_n)}g(X_n[\rho])\Biggr)^2\Biggr]\Biggr|\\
&\le\|g\|_\infty^2\E\biggl[\biggl|\frac{1_{\{f_k(X_n)>0\}}}{\P(f_k(X_n)>0)}-\biggl(\frac{f_k(X_n)}{\binom n{k+1}q_k}\biggr)^2\biggr|\biggr]\\
&\le\|g\|_\infty^2\biggl(\E\biggl[\biggl|\frac{1_{\{f_k(X_n)>0\}}}{\P(f_k(X_n)>0)}-1\biggr|\biggr]+\E\biggl[\biggl|\biggl(\frac{f_k(X_n)}{\binom n{k+1}q_k}\biggr)^2-1\biggr|\biggr]\biggr). 
\end{align*}
From Lemma~\ref{lem:num_faces}, the last line above converges to zero as $n\to\infty$. Thus, combining this estimate with Eq.~\eqref{eq:siki4} yields Eq.~\eqref{eq:2moment}. 
\end{proof}

\section{Convergence of Betti numbers and empirical spectral distributions}
\subsection{Statement of the result}\label{ssec:mainBetti}
In this section, we consider homogeneous and spatially independent random subcomplexes of $\triangle_n$ and study the asymptotic behavior of their Betti numbers and the empirical spectral distributions of their Laplacians as $n$ tends to infinity. Recall the definitions of the parameters $q_k$ and $r_k$ as described in Eq.~\eqref{eq:para}.
The following theorem is the main result in this section.
\begin{theorem}\label{thm:LT_B_HSIRSC}
Let $k\ge0$ and $c>0$ be fixed, and let $X_n$ be a homogeneous and spatially independent random subcomplex of $\triangle_n$. If $n^{k+1}q_k=\om(1)$ and $nr_k\sim c$, then the following $(1)$ and $(2)$ hold. 
\begin{enumerate}
\item For any $r\in[1,\infty)$,
\[
\lim_{n\to\infty}\E\biggl[\biggl|\frac{\b_k(X_n)}{n^{k+1}q_k}-\frac{h_k(c)}{(k+1)!}\biggr|^r\biggr]=0. 
\]
\item There exists a deterministic probability measure $\mu\in\cP_\R$ such that for any open set $U\subset\cP_\R$ such that $\mu\in U$, 
\[
\lim_{n\to\infty}\P\bigl(\mu_{L_k^{\up}(X_n)}\in U\mid\dim X_n\ge k\bigr)=1. 
\]
In other words, the empirical spectral distribution $\mu_{L_k^{\up}(X_n)}$ under $\P(\cdot\mid\dim X_n\ge k)$ converges weakly to $\mu$ in distribution as $n\to\infty$.
\end{enumerate}
\end{theorem}
\begin{cor}\label{cor:LT_B_HSIRSC}
Let $k\ge1$ and $c>0$ be fixed, and let $X_n$ be a homogeneous and spatially independent random subcomplex of $\triangle_n$. If $n^{k+1}q_k=\om(1)$ and $nr_{k-1}\sim c$, then for any $r\in[1,\infty)$,
\[
\lim_{n\to\infty}\E\biggl[\biggl|\frac{\b_k(X_n)}{n^{k+1}q_k}-\frac{g_k(c)}{(k+1)!}\biggr|^r\biggr]=0,
\]
where
\[
g_k(c)\coloneqq\frac{k+1}c\biggl(h_{k-1}(c)-\biggl(1-\frac c{k+1}\biggr)\biggr).
\]
\end{cor}
\begin{rem}
From Eq.~\eqref{eq:h},
\begin{align*}
g_k(c)
&=\max\biggl\{0,(1-t_{k,c})^{k+1}-\frac{k+1}c(1-t_{k,c})+(k+1)t_{k,c}(1-t_{k,c})^k\biggr\}\\
&=\begin{cases}
0																		&(0\le c\le c_k),\\
(1-t_{k,c})^{k+1}-\frac{k+1}c(1-t_{k,c})+(k+1)t_{k,c}(1-t_{k,c})^k>0				&(c>c_k).
\end{cases}
\end{align*}
Therefore, under the assumption of Corollary~\ref{cor:LT_B_HSIRSC}, if $c\le c_k$, then for any $\eps>0$, with probability tending to one as $n\to\infty$, $\b_k(X_n)\le\eps n^{k+1}q_k$ holds.
Meanwhile, if $c>c_k$, then there exists $\eps_1>\eps_0>0$ such that, with probability tending to one as $n\to\infty$, $\eps_1n^{k+1}q_k\ge\b_k(X_n)\ge\eps_0n^{k+1}q_k$ holds.
In particular, $c>c_k$ implies that $H^k(X_n)$ is nontrivial with probability tending to one as $n\to\infty$.
\end{rem}
We apply Theorem~\ref{thm:LT_B_HSIRSC}~(1) and Corollary~\ref{cor:LT_B_HSIRSC} to several typical random simplicial complex models.
\begin{example}[$d$-Linial--Meshulam complex]\label{ex:LT_B_LMC}
Let $d\in\N$ and $c>0$. Consider a $d$-Linial--Meshulam complex $Y_n\sim Y_d(n,p)$ with $p=c/n$. Note that $q_{d-1}=1$, $q_d=p$, and $r_{d-1}=p$ for $n\ge d+1$ using Eq.~\eqref{eq:paraMP} with Example~\ref{ex:dLMC}. Then, we obtain $n^dq_{d-1}=n^d=\om(1)$, $n^{d+1}q_d=n^dc=\om(1)$, and $nr_{d-1}=c$. Therefore, Theorem~\ref{thm:LT_B_HSIRSC}~(1) with $k=d-1$ and Corollary~\ref{cor:LT_B_HSIRSC} with $k=d$ together imply that for any $r\in[1,\infty)$, 
\begin{gather*}
\lim_{n\to\infty}\E\biggl[\biggl|\frac{\b_{d-1}(Y_n)}{n^d}-\frac{h_{d-1}(c)}{d!}\biggr|^r\biggr]=0
\shortintertext{and}
\lim_{n\to\infty}\E\biggl[\biggl|\frac{\b_d(Y_n)}{n^dc}-\frac{g_d(c)}{(d+1)!}\biggr|^r\biggr]=0. 
\end{gather*}
The first equation above implies Theorem~\ref{thm:LT_B_LMC}.
\end{example}
\begin{example}[Random $d$-clique complex]\label{ex:LT_B_RCC}
Let $d\in\N$, $k\ge d-1$, and $c>0$. Consider the random $d$-clique complex $C_n\sim C_d(n,p)$ with $p=(c/n)^{1/\binom{k+1}d}$. Note that $q_k=p^{\binom{k+1}{d+1}}$, $q_{k+1}=p^{\binom{k+2}{d+1}}$, and $r_k=p^{\binom{k+1}d}$ for $n\ge k+2$ using Eq.~\eqref{eq:paraMP} with Example~\ref{ex:RdCC}. Here, $\binom d{d+1}=0$ by convention. Then, we obtain $n^{k+1}q_k=n^{(k+2)d/(d+1)}c^{(k+1-d)/(d+1)}=\om(1)$, $n^{k+2}q_{k+1}=n^{(k+2)d/(d+1)}c^{(k+2)/(d+1)}=\om(1)$, and $nr_k=c$.
Therefore, Theorem~\ref{thm:LT_B_HSIRSC}~(1) and Corollary~\ref{cor:LT_B_HSIRSC} with $k+1$ instead of $k$ together imply that for any $r\in[1,\infty)$, 
\begin{gather*}
\lim_{n\to\infty}\E\biggl[\biggl|\frac{\b_k(C_n)}{n^{(k+2)d/(d+1)}c^{(k+1-d)/(d+1)}}-\frac{h_k(c)}{(k+1)!}\biggr|^r\biggr]=0
\shortintertext{and}
\lim_{n\to\infty}\E\biggl[\biggl|\frac{\b_{k+1}(C_n)}{n^{(k+2)d/(d+1)}c^{(k+2)/(d+1)}}-\frac{g_{k+1}(c)}{(k+2)!}\biggr|^r\biggr]=0.
\end{gather*}
When $d=1$, the first equation above immediately implies Theorem~\ref{thm:LT_B_RCC}.
\end{example}
Before proceeding to the proof, we relate Theorem~\ref{thm:LT_B_HSIRSC}~(1) and Corollary~\ref{cor:LT_B_HSIRSC} to previously obtained results on the asymptotic behavior of Betti numbers in the critical dimension of multi-parameter random simplicial complexes.
Let $X\sim X(n,\p)$ be a multi-parameter random simplicial complex with $\p=(p_0,p_1,\ldots,p_{n-1})$.
As in the previous study~\cite{CF17b}, we write $p_i=n^{-\a_i}$ using a multi-exponent $\a=(\a_i)_{i\ge0}$ with $\a_i\ge0$, and additionally assume that $\a_i$'s do not depend on $n$ for simplicity.
Below, we use some notation in~\cite{CF17b} in a way that is useful for our purpose. For $k\ge0$, define
\[
\tau_k(\a)\coloneqq k+1-\sum_{i=0}^k\binom{k+1}{i+1}\a_i\quad\text{and}\quad\psi_k(\a)\coloneqq\sum_{i=0}^k\binom ki\a_i.
\]
Here, we set $\binom00=1$ by convention. Clearly, $\psi_k(\a)$ is nondecreasing with respect to $k$.
Consider the following disjoint convex domains in the space of multi-exponents:
\[
\mathfrak{D}_k\coloneqq\{\a=(\a_i)_{i\ge0}\mid\psi_k(\a)<1<\psi_{k+1}(\a)\}
\]
for $k\ge0$. When $\a=(\a_i)_{i\ge0}\in\mathfrak{D}_{k_*}$ for some integer $k_*\ge0$, the critical dimension of the multi-parameter random simplicial complex $X$ is said to equal $k_*$. Costa and Farber showed the following \textit{homological domination principle} which states that the Betti number in the critical dimension is significantly larger than any other Betti numbers (see~\cite[Theorem~4.3]{CF17b} for a stronger result).
\begin{theorem}[{\cite[Theorem~4.3]{CF17b}}]\label{thm:HDP}
Let $k\ge0$ be fixed, and let $X\sim X(n,\p)$ be a multi-parameter random simplicial complex with the above setting. If the critical dimension of $X$ equals $k$, that is, $\a\in\mathfrak{D}_k$, then the following $(1)$ and $(2)$ hold.
\begin{enumerate}
\item For any $\eps>0$, with probability tending to one,
\[
(1-\eps)\frac{n^{\tau_k(\a)}}{(k+1)!}\le\b_k(X)\le(1+\eps)\frac{n^{\tau_k(\a)}}{(k+1)!}.
\]
\item There exists $C\ge0$ depending only on $k$, such that for any $j\neq k$, with probability tending to one,
\[
\b_j(X)\le Cn^{-e(\a)}\b_k(X),
\]
where $e(\a)\coloneqq\min_{i\ge0}\{|1-\psi_i(\a)|\}=\min\{1-\psi_k(\a),\psi_{k+1}(\a)-1\}>0$.
\end{enumerate}
\end{theorem}
On the other hand, since
\[
n^{k+1}q_k=n^{k+1}\prod_{i=0}^k p_i^{\binom{k+1}{i+1}}=n^{\tau_k(\a)}
\quad\text{and}\quad
nr_k=n\prod_{i=0}^{k+1}p_i^{\binom{k+1}i}=n^{1-\psi_{k+1}(\a)}
\]
for $0\le k\le n-2$ from Eq.~\eqref{eq:paraMP}, Theorem~\ref{thm:LT_B_HSIRSC}~(1) and Corollary~\ref{cor:LT_B_HSIRSC} can be interpreted as dealing with the asymptotic behavior of the proportion of $\b_k(X)$ to $n^{\tau_k(\a)}$ on hyperplanes
\[
H_{k+1}=\{\a=(\a_i)_{i\ge0}\mid\psi_{k+1}(\a)=1\}
\quad\text{and}\quad
H_k=\{\a=(\a_i)_{i\ge0}\mid\psi_k(\a)=1\},
\]
respectively (cf. Theorem~\ref{thm:HDP}~(1)). These hyperplanes $H_{k+1}$ and $H_k$ together create the boundary of $\mathfrak{D}_k$ in the space of multi-exponents. As seen in Examples~\ref{ex:LT_B_LMC} and~\ref{ex:LT_B_RCC}, the Betti number in one specific dimension does not necessarily dominates the Betti numbers in all other dimensions on such boundaries, unlike the situation in Theorem~\ref{thm:HDP}~(2). Informally speaking, the critical dimension of the multi-parameter random simplicial complex gradually transitions from $k$ to $k+1$ on $H_{k+1}$.
See also~\cite[Sections~8 and~9]{CF17b},~\cite{Fo},~\cite[Section~3]{HK} for the homological properties of multi-parameter random simplicial complexes below or above the critical dimension.
%

\subsection{Proof of Theorem~\ref{thm:LT_B_HSIRSC}~$(2)$ and upper estimate of the Betti number}
Theorem~\ref{thm:LWLT_HSIRSC} is critical for the proof of Theorem~\ref{thm:LT_B_HSIRSC}. Our approach is essentially according to the idea used in~\cite{LP1},~\cite{LP2} for the proof of Theorem~\ref{thm:LT_B_LMC}.
In what follows, we always fix $k\ge0$ and $c>0$, and let $X_n$ be a homogeneous and spatially independent random subcomplex of $\triangle_n$ such that $n^{k+1}q_k=\om(1)$ and $nr_k\sim c$.

For proving Theorem~\ref{thm:LT_B_HSIRSC}~(2), we introduce some additional notation. Let $\cS'_k$ denote the set of all $[X,\tau]\in\cS_k$ such that $L_k^{\up}(X)$ is an essentially self-adjoint operator. It is a fact that, $\P'$-almost surely, $[\pt_k(c),\tau_o]\in\cS'_k$~(see, e.g.,~\cite[Claim~3.3]{LP1}). We define a kernel $M_k\colon\cS'_k\times\cB_\R\to[0,1]$ by $M_k([X,\tau],B)\coloneqq\mu_{(X,\tau)}(B)$ for $[X,\tau]\in\cS'_k$ and $B\in\cB_\R$. In fact, $\cS'_k\ni[X,\tau]\mapsto M_k([X,\tau],\cdot)\in\cP_\R$ is continuous because taking the rooted spectral measure is continuous~(cf.~\cite[Lemma~3.2]{LP1}). The proof of Theorem~\ref{thm:LT_B_HSIRSC} follows from Theorem~\ref{thm:LWLT_HSIRSC} using a map $\times M_k\colon\cP_{\cS'_k}\to\cP_\R$ defined by
\[
(\times M_k)(\nu)\coloneqq\nu M=\int_{\cS'_k}\nu(d[X,\tau])M([X,\tau],\cdot)\quad\text{for }\nu\in\cP_{\cS'_k}. 
\]
It is easy to confirm that the map $\times M_k$ is also continuous. 
\begin{proof}[Proof of Theorem~\ref{thm:LT_B_HSIRSC}~$(2)$]
Given the event $\{\dim X_n\ge k\}$,
\[
\mu_{L_k^{\up}(X_n)}=\frac1{f_k(X_n)}\sum_{\tau\in F_k(X_n)}\mu_{(X_n,\tau)}=\frac1{f_k(X_n)}\sum_{\tau\in F_k(X_n)}M_k(X_n[\tau],\cdot)=\lm_k(X_n)M_k. 
\]
The first identity follows from Eq.~\eqref{eq:SpeMeas}. In the second identity, we use $\mu_{X_n(\tau)}=\mu_{(X_n,\tau)}$ for any $\tau\in F_k(X_n)$. From combining Theorem~\ref{thm:LWLT_HSIRSC} and the continuous mapping theorem, $\lm_k(X_n)M_k$ under $\P(\cdot\mid\dim X_n\ge k)$ converges weakly to $\nu_k(c)M_k$ in distribution as $n\to\infty$. 
\end{proof}
\begin{rem}\label{rem:LT_B_HSIRSC}
From the above proof, the deterministic probability measure $\mu$ in Theorem~\ref{thm:LT_B_HSIRSC}~(2) can be expressed using the $k$-rooted Poisson tree $(\pt_k(c),\tau_o)$:
\[
\mu=\nu_k(c)M_k=\E'[M_k([\pt_k(c),\tau_o],\cdot)]=\E'\bigl[\mu_{(\pt_k(c),\tau_o)}\bigr]. 
\]
From the recursive structure of the $k$-rooted Poisson tree, Linial and Peled~\cite{LP1} provided the following upper estimate of $\E'\bigl[\mu_{(\pt_k(c),\tau_o)}(\{0\})\bigr]$ (see also~\cite[Section~5]{LP2}): 
\begin{align*}
&\E'\bigl[\mu_{(\pt_k(c),\tau_o)}(\{0\})\bigr]\\
&\le\max\left\{t+ct(1-t)^{k+1}-\frac c{k+2}\bigl(1-(1-t)^{k+2}\bigr)\relmiddle| t\in[0,1],t=\exp\bigl(-c(1-t)^{k+1}\bigr)\right\}
=h_k(c).
\end{align*}
\end{rem}
The upper estimate of the Betti number follows immediately from Theorem~\ref{thm:LT_B_HSIRSC}~(2) and Remark~\ref{rem:LT_B_HSIRSC}. 
\begin{prop}\label{prop:UB_B}
Let $\eps>0$ be fixed. Then, 
\[
\lim_{n\to\infty}\P\left(\frac{\b_k(X_n)}{f_k(X_n)}>h_k(c)+\eps\relmiddle|\dim X_n\ge k\right)=0. 
\]
\end{prop}
\begin{proof}
Given the event $\{\dim X_n\ge k\}$, a simple calculation yields
\[
\mu_{L_k^{\up}(X_n)}(\{0\})=\frac{\dim\bigl(\ker L_k^{\up}(X_n)\bigr)}{f_k(X_n)}=\frac{\dim Z^k(X_n)}{f_k(X_n)}. 
\]
Therefore, from Theorem~\ref{thm:LT_B_HSIRSC}~(2) and Remark~\ref{rem:LT_B_HSIRSC}, we obtain
\begin{align*}
&\limsup_{n\to\infty}\P\left(\frac{\b_k(X_n)}{f_k(X_n)}\ge h_k(c)+\eps\relmiddle|\dim X_n\ge k\right)\\
&\le\limsup_{n\to\infty}\P\left(\frac{\dim Z^k(X_n)}{f_k(X_n)}\ge\E'\bigl[\mu_{(\pt_k(c),\tau_o)}(\{0\})\bigr]+\eps\relmiddle|\dim X_n\ge k\right)\\
&=\limsup_{n\to\infty}\P\bigl(\mu_{L_k^{\up}(X_n)}(\{0\})\ge\E'\bigl[\mu_{(\pt_k(c),\tau_o)}\bigr](\{0\})+\eps\mid\dim X_n\ge k\bigr)\\
&=0.
\end{align*}
In the last line above, we also use the fact that the map $\cP_\R\ni\mu\mapsto\mu(\{0\})\in\R$ is upper semi-continuous. Thus, the conclusion follows. 
\end{proof}

\subsection{Lower estimate of the Betti number}
The following inequality is a simple lower estimate of the Betti number of a given finite simplicial complex.
\begin{prop}[{A version of the Morse inequality}]\label{prop:MorseIneq}
Let $X$ be a finite simplicial complex. Then, it holds that
\[
\b_k(X)\ge f_k(X)-f_{k+1}(X)-f_{k-1}(X). 
\]
\end{prop}
\begin{proof}
Since $f_k(X)=\dim Z^k(X)+\dim B^{k+1}(X)$, we have
\begin{align*}
\b_k(X)&=\dim Z^k(X)-\dim B^k(X)\\
&=\bigl(f_k(X)-\dim B^{k+1}(X)\bigr)-\dim B^k(X)\\
&\ge f_k(X)-f_{k+1}(X)-f_{k-1}(X). \qedhere
\end{align*}
\end{proof}
The following lemma follows from Proposition~\ref{prop:MorseIneq}.
\begin{lemma}\label{lem:Morselower}
Let $\eps>0$ be fixed. Then, it holds that
\[
\lim_{n\to\infty}\P\left(\frac{\b_k(X_n)}{f_k(X_n)}<1-\frac c{k+2}-\eps\relmiddle|\dim X_n\ge k\right)=0. 
\]
\end{lemma}
\begin{proof}
From Proposition~\ref{prop:MorseIneq}, we have
\begin{align*}
&\P\left(\frac{\b_k(X_n)}{f_k(X_n)}<1-\frac c{k+2}-\eps\relmiddle|\dim X_n\ge k\right)\\
&\le\P\left(\frac{f_{k+1}(X_n)}{f_k(X_n)}-\frac c{k+2}>\eps/2\relmiddle|\dim X_n\ge k\right)+\P\left(\frac{f_{k-1}(X_n)}{f_k(X_n)}>\eps/2\relmiddle|\dim X_n\ge k\right). 
\end{align*}
Now, given the event $\{\dim X_n\ge k\}$, 
\begin{align*}
\frac{f_{k+1}(X_n)}{f_k(X_n)}=\frac{f_{k+1}(X_n)}{\binom n{k+2}q_{k+1}}\frac{\binom n{k+1}q_k}{f_k(X_n)}\frac{\binom n{k+2}r_k}{\binom n{k+1}}
\shortintertext{and}
\frac{f_{k-1}(X_n)}{f_k(X_n)}=\frac{f_{k-1}(X_n)}{\binom nkq_{k-1}}\frac{\binom n{k+1}q_k}{f_k(X_n)}\frac{\binom nk}{\binom n{k+1}r_k}s_k. 
\end{align*}
Note that $n^{k+2}q_{k+1}=\om(1)$ because $n^{k+1}q_k=\om(1)$ and $nr_k\sim c$. Therefore, combining Lemmas~\ref{lem:qrs} and~\ref{lem:num_faces}, we obtain
\begin{align}
\lim_{n\to\infty}\P\left(\biggl|\frac{f_{k+1}(X_n)}{f_k(X_n)}-\frac c{k+2}\biggr|>\eps/2\relmiddle|\dim X_n\ge k\right)&=0\label{eq:ff_1}
\shortintertext{and}
\lim_{n\to\infty}\P\left(\frac{f_{k-1}(X_n)}{f_k(X_n)}>\eps/2\relmiddle|\dim X_n\ge k\right)&=0. \label{eq:ff_2}
\end{align}
These complete the proof. 
\end{proof}

Lemma~\ref{lem:Morselower} gives a simple and useful lower bound on the asymptotic behavior of the $k$th Betti number of $X_n$. However, there is still room for improving the lower bound. To do that, we use the number of (inclusion-wise) maximal $k$-simplices in $X_n$ after some collapsing procedures.
Let $X$ be a simplicial complex. A simplex $\tau$ in $X$ is said to be \textit{free} if there exists a unique maximal simplex $\sg_\tau$ in $X$, strictly containing $\tau$. The removal of all the simplices $\eta$ in $X$ such that $\tau\subset\eta\subset\sg_\tau$ is called a \textit{collapse}. Moreover, when $\dim\sg_\tau=\dim\tau+1$, we call the collapse an \textit{elementary collapse}. We then define a collapsing operator $R_k$ as follows. We first list all the maximal $(k+1)$-simplices $\sg$ in $X$ containing at least one free $k$-simplex and remove those $\sg$'s from $X$ together with an arbitrarily chosen free $k$-dimensional face of $\sg$. We denote the resulting subcomplex of $X$ by $R_k(X)$. Note that $X$ and $R_k(X)$ are homotopy equivalent. We also define $R_k^0(X)\coloneqq X$ and $R_k^{l+1}(X)\coloneqq R_k\bigl(R_k^l(X)\bigr)$ for $l\ge0$. Furthermore, we define $S_k^l(X)$ by removing all the maximal $k$-simplices from $R_k^l(X)$. 

Now, let $l\ge0$ be fixed. When $X$ is finite, we have $f_k\bigl(R_k^l(X)\bigr)=f_k\bigl(S_k^l(X)\bigr)+I_k\bigl(R_k^l(X)\bigr)$, where $I_k\bigl(R_k^l(X)\bigr)$ denotes the number of maximal $k$-simplices of $R_k^l(X)$. In addition,
\[
f_k(X)-f_k\bigl(R_k^l(X)\bigr)=f_{k+1}(X)-f_{k+1}(R_k^l(X))=f_{k+1}(X)-f_{k+1}\bigl(S_k^l(X)\bigr). 
\]
Combining these equations, we obtain
\[
I_k\bigl(R_k^l(X)\bigr)=f_k(X)-f_{k+1}(X)+f_{k+1}\bigl(S_k^l(X)\bigr)-f_k\bigl(S_k^l(X)\bigr). 
\]
Therefore,
\begin{align}\label{eq:Bettilower}
\b_k(X)&=\b_k\bigl(R_k^l(X)\bigr)\nonumber\\
&=\dim Z^k\bigl(R_k^l(X)\bigr)-\dim B^k\bigl(R_k^l(X)\bigr)\nonumber\\
&\ge \dim Z^k\bigl(R_k^l(X)\bigr)-f_{k-1}\bigl(R_k^l(X)\bigr)\nonumber\\
&=\dim Z^k(R_k^l(X))-f_{k-1}(X)\nonumber\\
&\ge I_k\bigl(R_k^l(X)\bigr)-f_{k-1}(X)\nonumber\\
&=f_k(X)-f_{k+1}(X)-f_{k-1}(X)+f_{k+1}\bigl(S_k^l(X)\bigr)-f_k\bigl(S_k^l(X)\bigr)\nonumber\\
&=f_k(X)-f_{k+1}(X)-f_{k-1}(X)+\sum_{\tau\in F_k(X)}1_{\{\tau\in S_k^l(X)\}}\biggl(\frac{\deg\bigl(S_k^l(X);\tau\bigr)}{k+2}-1\biggr). 
\end{align}
The last line follows from a double counting argument:
\[
(k+2)f_{k+1}\bigl(S_k^l(X)\bigr)=\sum_{\tau\in F_k\bigl(S_k^l(X)\bigr)}\deg\bigl(S_k^l(X);\tau\bigr).
\]
We now define a map $D_k^{(l)}\colon\cS_k\to\R$ by
\[
D_k^{(l)}([X,\tau])\coloneqq1_{\{\tau\in S_k^l(X)\}}\biggl(\frac{\deg\bigl(S_k^l(X);\tau\bigr)}{k+2}-1\biggr). 
\]
Suppose that $([Y_n,\tau_n])_{n=1}^\infty$ is a convergent sequence to $[Y,\tau]$ in $\cS_k$. By the definition of the local distance, there exists $N\in\N$ such that $n\ge N$ implies that $(Y_n,\tau_n)_{l+1}\simeq(Y,\tau)_{l+1}$. Then, we have $D_k^{(l)}([Y_n,\tau_n])=D_k^{(l)}([Y,\tau])$ for $n\ge N$ because the map $D_k^{(l)}$ depends on only the simplices of distance at most $l+1$ from the root. This means that $D_k^{(l)}$ is continuous. The following lower bound on the Betti number follows from Theorem~\ref{thm:LWLT_HSIRSC} using a map $\times D_k^{(l)}\colon\cP_{\cS_k}\to\R\cup\{\infty\}$ defined by
\[
\bigl(\times D_k^{(l)}\bigr)(\nu)\coloneqq\nu D_k^{(l)}=\int_{\cS_k}\nu(d[X,\tau])D_k^{(l)}([X,\tau])\quad\text{for }\nu\in\cP_{\cS_k}. 
\]
It is easy to confirm that the map $\times D_k^{(l)}$ is lower semi-continuous.
\begin{lemma}\label{lem:lowerBetti}
Let $l\ge0$ and $\eps>0$ be fixed. Then, it holds that
\[
\lim_{n\to\infty}\P\left(\frac{\b_k(X_n)}{f_k(X_n)}<1-\frac c{k+2}+\E'\bigl[D_k^{(l)}([\pt_k(c),\tau_o])\bigr]-\eps\relmiddle|\dim X_n\ge k\right)=0. 
\]
\end{lemma}
\begin{proof}
From Eq.~\eqref{eq:Bettilower}, given the event $\{\dim X_n\ge k\}$, 
\begin{align}\label{eq:eq:lowerBetti}
\frac{\b_k(X_n)}{f_k(X_n)}
&\ge1-\frac{f_{k+1}(X_n)}{f_k(X_n)}-\frac{f_{k-1}(X_n)}{f_k(X_n)}+\frac1{f_k(X_n)}\sum_{\tau\in F_k(X_n)}1_{\{\tau\in S_k^l(X_n)\}}\biggl(\frac{\deg\bigl(S_k^l(X_n);\tau\bigr)}{k+2}-1\biggr)\nonumber\\
&=1-\frac{f_{k+1}(X_n)}{f_k(X_n)}-\frac{f_{k-1}(X_n)}{f_k(X_n)}+\frac1{f_k(X_n)}\sum_{\tau\in F_k(X_n)}D_k^{(l)}(X_n[\tau])\nonumber\\
&=1-\frac{f_{k+1}(X_n)}{f_k(X_n)}-\frac{f_{k-1}(X_n)}{f_k(X_n)}+\lm_k(X_n)D_k^{(l)}. 
\end{align}
From Theorem~\ref{thm:LWLT_HSIRSC} and the lower semi-continuity of the map $\times D_k^{(l)}$, we have
\begin{align}\label{eq:LB_B}
&\limsup_{n\to\infty}\P\bigl(\lm_k(X_n)D_k^{(l)}\le\E'\bigl[D_k^{(l)}([\pt_k(c),\tau_o])\bigr]-\eps\mid\dim X_n\ge k\bigr)\nonumber\\
&=\limsup_{n\to\infty}\P\bigl(\lm_k(X_n)D_k^{(l)}\le\nu_k(c)D_k^{(l)}-\eps\mid\dim X_n\ge k\bigr)\nonumber\\
&=0. 
\end{align}
Thus, the conclusion follows from Eqs.~\eqref{eq:ff_1},~\eqref{eq:ff_2},~\eqref{eq:eq:lowerBetti}, and~\eqref{eq:LB_B} in the same manner as the proof of Lemma~\ref{lem:Morselower}. 
\end{proof}

We now give an overview of the estimate of $\E'\bigl[D_k^{(l)}([\pt_k(c),\tau_o])\bigr]$ as described in~\cite[Section~3]{ALLM}. To provide the estimate, we introduce the concept of \textit{$k$-rooted tree pruning}. For a $k$-rooted tree $(T,\tau)$, we define the pruning $Q_k((T,\tau))$ as below. Initially, let $\{\tau_1,\tau_2,\ldots\}$ be the set of all the free $k$-simplices in $T$ that are distinct from $\tau$, and we take the unique simplex $\sg_i\in F_{k+1}(T)$ containing $\tau_i$. We then define $\tilde Q_k(T)$ as a simplicial complex obtained from $T$ by removing all the simplices $\tau_j$ and $\sg_j$~($j=1,2,\ldots$). Finally, define $Q_k((T,\tau))$ as the $k$-rooted tree $\bigl(\tilde Q_k(T)\bigr)(\tau)$. Furthermore, we define $Q_k^0((T,\tau))\coloneqq(T,\tau)$ and $Q_k^{l+1}((T,\tau))\coloneqq Q_k\bigl(Q_k^l((T,\tau))\bigr)$ for $l\ge0$. 
A straightforward calculation gives the following lemma.
\begin{lemma}[{\cite[Lemma~3.3]{LP2}}]
Let $\bigl(t_{k+1,c}^{(l)}\bigr)_{l\ge-1}$ be a sequence defined by
\[
t_{k+1,c}^{(-1)}=0\quad\text{and}\quad t_{k+1,c}^{(l+1)}=\exp\bigl(-c\bigl(1-t_{k+1,c}^{(l)}\bigr)^{k+1}\bigr)\quad\text{for }l\ge-1.
\]
Furthermore, set $\dl_{k+1,c}^{(l)}\coloneqq\deg\bigl(Q_k^l((\pt_k(c),\tau_o));\tau_o\bigr)$. Then, $\dl_{k+1,c}^{(l)}$ follows the Poisson distribution with parameter $c\bigl(1-t_{k+1,c}^{(l-1)}\bigr)^{k+1}$ for every $l\ge0$. 
\end{lemma}
The following lemma gives a lower estimate of $\E'\bigl[D_k^{(l)}([\pt_k(c),\tau_o])\bigr]$ using the values $\bigl(t_{k+1,c}^{(l)}\bigr)_{l\ge-1}$.
\begin{lemma}[{\cite[Section~5.2]{LP2}}]\label{lem:lower}
For $l\ge1$, 
\begin{gather*}
\P'\bigl(\tau_o\in S_k^l(\pt_k(c))\bigr)
\le1-t_{k+1,c}^{(l-1)}-c\bigl(1-t_{k+1,c}^{(l-2)}\bigr)^{k+1}t_{k+1,c}^{(l-1)}
\shortintertext{and}
\E'\bigl[\deg\bigl(S_k^l(\pt_k(c));\tau_o\bigr);\tau_o\in S_k^l(\pt_k(c))\bigr]
\ge c\bigl(1-t_{k+1,c}^{(l-1)}\bigr)^{k+1}\bigl(1-t_{k+1,c}^{(l)}\bigr). 
\end{gather*}
In particular, 
\[
\E'\bigl[D_k^{(l)}([\pt_k(c),\tau_o])\bigr]\ge\frac c{k+2}\bigl(1-t_{k+1,c}^{(l-1)}\bigr)^{k+1}\bigl(1-t_{k+1,c}^{(l)}\bigr)-1+t_{k+1,c}^{(l-1)}+c\bigl(1-t_{k+1,c}^{(l-2)}\bigr)^{k+1}t_{k+1,c}^{(l-1)}. 
\]
\end{lemma}
\begin{proof}
Since $\bigl\{\tau_o\in S_k^l(\pt_k(c))\bigr\}\subset\bigl\{\dl_{k+1,c}^{(l-1)}\ge2\bigr\}$, we have
\begin{align*}
\P'\bigl(\tau_o\in S_k^l(\pt_k(c))\bigr)
&\le\P'\bigl(\dl_{k+1,c}^{(l-1)}\ge2\bigr)\\
&=1-\exp\bigl(-c\bigl(1-t_{k+1,c}^{(l-2)}\bigr)^{k+1}\bigr)-c\bigl(1-t_{k+1,c}^{(l-2)}\bigr)^{k+1}\exp\bigl(-c\bigl(1-t_{k+1,c}^{(l-2)}\bigr)^{k+1}\bigr)\\
&=1-t_{k+1,c}^{(l-1)}-c\bigl(1-t_{k+1,c}^{(l-2)}\bigr)^{k+1}t_{k+1,c}^{(l-1)}. 
\end{align*}
Furthermore, we note that given the event $\bigl\{\dl_{k+1,c}^{(l)}\ge2\bigr\}$, 
\[
1_{\{\tau_o\in S_k^l(\pt_k(c))\}}\deg\bigl(S_k^l(\pt_k(c));\tau_o\bigr)=\dl_{k+1,c}^{(l)}.
\]
Therefore,
\begin{align*}
&\E'\bigl[\deg\bigl(S_k^l(\pt_k(c));\tau_o\bigr);\tau_o\in S_k^l(\pt_k(c))\bigr]\\
&\ge\E'\bigl[1_{\{\tau_o\in S_k^l(\pt_k(c))\}}\deg\bigl(S_k^l(\pt_k(c));\tau_o\bigr);\dl_{k+1,c}^{(l)}\ge2\bigr]\\
&=\E'\bigl[\dl_{k+1,c}^{(l)};\dl_{k+1,c}^{(l)}\ge2\bigr]\\
&=c\bigl(1-t_{k+1,c}^{(l-1)}\bigr)^{k+1}-c\bigl(1-t_{k+1,c}^{(l-1)}\bigr)^{k+1}\exp\bigl(-c\bigl(1-t_{k+1,c}^{(l-1)}\bigr)^{k+1}\bigr)\\
&=c\bigl(1-t_{k+1,c}^{(l-1)}\bigr)^{k+1}\bigl(1-t_{k+1,c}^{(l)}\bigr). \qedhere
\end{align*}
\end{proof}
Now, we define
\[
h_k^{(l)}(c)\coloneqq\max\biggl\{1-\frac c{k+2},t_{k+1,c}^{(l-1)}+ct_{k+1,c}^{(l-1)}\bigl(1-t_{k+1,c}^{(l-2)}\bigr)^{k+1}-\frac c{k+2}\bigl(1-\bigl(1-t_{k+1,c}^{(l-1)}\bigr)^{k+1}\bigl(1-t_{k+1,c}^{(l)}\bigr)\bigr)\biggr\}
\]
for $l\ge1$. Then, Lemma~\ref{lem:lower} implies that
\[
1-\frac c{k+2}+\bigl(0\vee\E'\bigl[D_k^{(l)}([\pt_k(c),\tau_o])\bigr]\bigr)\ge h_k^{(l)}(c)\quad\text{for }l\ge1.
\]
Therefore, the following proposition follows immediately from Lemmas~\ref{lem:Morselower} and~\ref{lem:lowerBetti}.
\begin{prop}\label{prop:lowerBetti}
Let $l\ge1$ and $\eps>0$ be fixed. Then, it holds that
\[
\lim_{n\to\infty}\P\left(\frac{\b_k(X_n)}{f_k(X_n)}<h_k^{(l)}(c)-\eps\relmiddle|\dim X_n\ge k\right)=0. 
\]
\end{prop}

\subsection{Proof of Theorem~\ref{thm:LT_B_HSIRSC}~$(1)$ and Corollary~\ref{cor:LT_B_HSIRSC}}
From the upper and lower bounds on the Betti number in Propositions~\ref{prop:UB_B} and~\ref{prop:lowerBetti}, respectively, we now prove the main result.
\begin{proof}[Proof of Theorem~\ref{thm:LT_B_HSIRSC}~$(1)$]
Let $\eps>0$ be fixed. We can take $l\in\N$ such that $|h_k^{(l)}(c)-h_k(c)|<\eps/2$ since $\lim_{l\to\infty}t_{k+1,c}^{(l)}=t_{k+1,c}$. Then, from Proposition~\ref{prop:lowerBetti}, we have
\begin{align*}
&\limsup_{n\to\infty}\P\left(\frac{\b_k(X_n)}{f_k(X_n)}<h_k(c)-\eps\relmiddle|\dim X_n\ge k\right)\\
&\le\limsup_{n\to\infty}\P\left(\frac{\b_k(X_n)}{f_k(X_n)}<h_k^{(l)}(c)-\eps/2\relmiddle|\dim X_n\ge k\right)=0. 
\end{align*}
Combining this estimate with Proposition~\ref{prop:UB_B}, we obtain
\begin{equation}\label{eq:LT_B_HSIRSC}
\lim_{n\to\infty}\P\left(\Bigl|\frac{\b_k(X_n)}{f_k(X_n)}-h_k(c)\Bigr|>\eps\relmiddle|\dim X_n\ge k\right)=0. 
\end{equation}
Furthermore, given the event $\{\dim X_n\ge k\}$, 
\begin{align}\label{eq:LT_B_HSIRSC1}
\biggl|\frac{\b_k(X_n)}{n^{k+1}q_k}-\frac{h_k(c)}{(k+1)!}\biggr|&\le\biggl|\frac{\b_k(X_n)}{n^{k+1}q_k}-\frac{\b_k(X_n)}{(k+1)!f_k(X_n)}\biggr|+\biggl|\frac{\b_k(X_n)}{(k+1)!f_k(X_n)}-\frac{h_k(c)}{(k+1)!}\biggr|\nonumber\\
&\le\biggl|\frac{f_k(X_n)}{n^{k+1}q_k}-\frac1{(k+1)!}\biggr|+\biggl|\frac{\b_k(X_n)}{f_k(X_n)}-h_k(c)\biggr|. 
\end{align}
For the second inequality above, we use a trivial upper bound $\b_k(X_n)\le f_k(X_n)$. By Eqs.~\eqref{eq:LT_B_HSIRSC} and~\eqref{eq:LT_B_HSIRSC1} together with Lemma~\ref{lem:num_faces}, a simple calculation provides the conclusion.
\end{proof}

\begin{proof}[Proof of Corollary~\ref{cor:LT_B_HSIRSC}]
Since
\begin{align*}
\b_i(X_n)=\dim Z^i(X_n)-\dim B^i(X_n)
\quad\text{and}\quad
f_i(X_n)=\dim Z^i(X_n)+\dim B^{i+1}(X_n)
\end{align*}
for $i\ge0$, we have
\[
\b_k(X_n)-\b_{k-1}(X_n)=f_k(X_n)-f_{k-1}(X_n)+\dim B^{k-1}(X_n)-\dim B^{k+1}(X_n).
\]
Therefore,
\begin{equation}\label{eq:corLT_B_HSIRSC1}
\frac{\b_k(X_n)}{n^{k+1}q_k}
=\frac{\b_{k-1}(X_n)}{(nr_{k-1})n^kq_{k-1}}+\frac{f_k(X_n)}{n^{k+1}q_k}-\frac{f_{k-1}(X_n)}{(nr_{k-1})n^kq_{k-1}}+\frac{\dim B^{k-1}(X_n)}{n^{k+1}q_k}-\frac{\dim B^{k+1}(X_n)}{n^{k+1}q_k}.
\end{equation}
Let $\eps>0$ be fixed. Applying Theorem~\ref{thm:LT_B_HSIRSC}~(1) with $k-1$ instead of $k$, we obtain
\begin{equation*}
\lim_{n\to\infty}\P\biggl(\biggl|\frac{\b_{k-1}(X_n)}{n^kq_{k-1}}-\frac{h_{k-1}(c)}{k!}\biggr|>\eps\biggr)=0.
\end{equation*}
In addition, since $n^{k+1}q_k=\om(1)$ and $n^kq_{k-1}=n^{k+1}q_k/(nr_{k-1})=\om(1)$ by the assumption, Lemma~\ref{lem:num_faces} yields
\begin{gather*}
\lim_{n\to\infty}\P\biggl(\biggl|\frac{f_k(X_n)}{n^{k+1}q_k}-\frac1{(k+1)!}\biggr|>\eps\biggr)=0
\shortintertext{and}
\lim_{n\to\infty}\P\biggl(\biggl|\frac{f_{k-1}(X_n)}{n^kq_{k-1}}-\frac1{k!}\biggr|>\eps\biggr)=0.
\end{gather*}
Next, we estimate the fourth and fifth terms in the right-hand side of Eq.~\eqref{eq:corLT_B_HSIRSC1}. Since $s_{k-1}=o(1)$ from Lemma~\ref{lem:qrs}, we have
\begin{gather*}
\frac{\E[\dim B^{k-1}(X_n)]}{n^{k+1}q_k}
\le\frac{\E[f_{k-2}(X_n)]}{n^{k+1}q_k}
\le\frac{n^{k-1}q_{k-2}}{n^{k+1}q_k}
=\frac1{(nr_{k-1})(nr_{k-2})}
=\frac{s_{k-1}}{(nr_{k-1})^2}
=o(1)
\shortintertext{and}
\frac{\E[\dim B^{k+1}(X_n)]}{n^{k+1}q_k}
\le\frac{\E[f_{k+1}(X_n)]}{n^{k+1}q_k}
\le\frac{n^{k+2}q_{k+1}}{n^{k+1}q_k}
=nr_k
=(nr_{k-1})s_k
\le(nr_{k-1})s_{k-1}
=o(1).
\end{gather*}
Therefore, from Markov's inequality, we have
\[
\lim_{n\to\infty}\P\biggl(\frac{\dim B^{k-1}(X_n)}{n^{k+1}q_k}>\eps\biggr)
=\lim_{n\to\infty}\P\biggl(\frac{\dim B^{k+1}(X_n)}{n^{k+1}q_k}>\eps\biggr)=0.
\]
Consequently, Eq.~\eqref{eq:corLT_B_HSIRSC1} together with all the estimates above implies that for any $\eps>0$,
\[
\lim_{n\to\infty}\P\biggl(\biggl|\frac{\b_k(X_n)}{n^{k+1}q_k}-\frac{g_k(c)}{(k+1)!}\biggr|>\eps\biggr)=0.
\]
Finally, using Lemma~\ref{lem:num_faces} with the trivial upper bound $\b_k(X_n)\le f_k(X_n)$, we can immediately conclude that for any $r\in[1,\infty)$,
\[
\lim_{n\to\infty}\E\biggl[\biggl|\frac{\b_k(X_n)}{n^{k+1}q_k}-\frac{g_k(c)}{(k+1)!}\biggr|^r\biggr]=0.\qedhere
\]
\end{proof}

\section*{Acknowledgements}
This study was supported by JSPS KAKENHI Grant Number 19J11237.
The author expresses gratitude to Professors Yasuaki~Hiraoka, Masanori~Hino, and Matthew~Kahle for their valuable comments.


%
%

\end{document}